\documentclass{amsart}
\usepackage{amssymb,graphicx}

\newtheorem{theorem}{Theorem}[section]
\newtheorem{lemma}[theorem]{Lemma}
\newtheorem{corollary}[theorem]{Corollary}
\newtheorem{proposition}[theorem]{Proposition}
\newtheorem{question}{Question}

\theoremstyle{definition}
\newtheorem{definition}[theorem]{Definition}
\theoremstyle{remark}
\newtheorem{remark}[theorem]{Remark}

\numberwithin{equation}{section}

\begin{document}

\title[Embedded Minimal Surfaces with Free Boundary]{A general existence theorem for embedded minimal surfaces with free boundary}

\author[Martin Li]{Martin Li}
\address{Mathematics Department, University of British Columbia, 1984 Mathematics Road, Vancouver, BC V6T 1Z2, Canada}
\email{martinli@math.ubc.ca}



\begin{abstract}
In this paper, we prove a general existence theorem for properly embedded minimal surfaces with free boundary in any compact Riemannian 3-manifold $M$ with boundary $\partial M$. These minimal surfaces are either disjoint from $\partial M$ or meet $\partial M$ orthogonally. The main feature of our result is that there is no assumptions on the curvature of $M$ or convexity of $\partial M$. We prove the boundary regularity of the minimal surfaces at their free boundaries. Furthermore, we define a topological invariant, the \emph{filling genus}, for compact 3-manifolds with boundary and show that we can bound the genus of the minimal surface constructed above in terms of the filling genus of the ambient manifold $M$. Our proof employs a variant of the min-max construction used by Colding and De Lellis on closed embedded minimal surfaces, which were first developed by Almgren and Pitts. 
\end{abstract}
\maketitle   



\tableofcontents




\section{Background and motivation}

In this paper, we study a general existence problem for embedded minimal surfaces with free boundary. All manifolds (with or without boundary) are assumed to be smooth up to the boundary unless otherwise stated.

\begin{question}
Given a compact Riemannian three-manifold $(M^3,g)$ with boundary $\partial M$, does there exist a \emph{properly embedded} minimal surface $\Sigma \subset M$ with boundary $\partial \Sigma \subset \partial M$ such that $\Sigma$ meets $\partial M$ orthogonally along $\partial \Sigma$?
\end{question}

Recall that a surface $\Sigma \subset M$ (with or without boundary) is said to be \emph{properly embedded} if the inclusion map $i:\Sigma \to M$ is a proper  embedding (i.e. a one-to-one immersion). In other words, we require
\begin{itemize}
\item[(i)] $\Sigma \cap \partial M=\partial \Sigma$; and
\item[(ii)] $\Sigma$ is transversal to $\partial M$ at any point on $\partial \Sigma$.
\end{itemize}
Note that if $\partial \Sigma = \emptyset$, this is equivalent to saying that $\Sigma$ is contained in the interior of $M$. 

The orthogonality condition along the boundary $\partial \Sigma$ is a natural condition arising variationally. Let $\Sigma \subset M$ be a properly embedded surface with boundary $\partial \Sigma \subset \partial M$. Suppose we have a smooth family of properly embedded surfaces $\{\Sigma_t\}_{t \in (-\epsilon,\epsilon)}$ in $M$ for some $\epsilon>0$ with $\Sigma_0=\Sigma$. If we calculate how the area of this family of surfaces, denoted by Area($\Sigma_t$), changes with respect to $t$ at $t=0$, a standard computation (for example, see \cite{Schoen06} and  \cite{Simon83}) gives the first variation formula
\begin{equation}
\left. \frac{d}{dt} \right|_{t=0} \text{Area}(\Sigma_t)=-\int_\Sigma \langle H,X \rangle \; da +\int_{\partial \Sigma} \langle X,\nu \rangle ds,
\end{equation}
where $\langle \; , \; \rangle$ is the metric of the ambient manifold $M$, $\nu$ is the outer conormal vector of $\partial \Sigma$ in $\Sigma$ (i.e. the outward unit normal of $\partial \Sigma$ tangent to $\Sigma$), $H$ is the mean curvature vector of $\Sigma$ in $M$ (with the sign convention that $H$ points inward for the unit sphere in $\mathbb{R}^3$) and $X$ is the variation field associated with the one-parameter family $\{\Sigma_t\}$ (i.e. $X(x)=\left. \frac{\partial}{\partial t} \right|_{t=0} F_t(x)$ where $F=F_t(x)=F(t,x):(-\epsilon,\epsilon) \times \Sigma \to M$ is a one-parameter family of proper embeddings such that $F_t(\Sigma)=\Sigma_t$). Note that since $\partial \Sigma_t \subset \partial M$ for all $t$, the variation field $X$ is tangential to $\partial M$ along $\partial \Sigma$. From (1.1),  $\Sigma$ is a critical point to the variational problem if and only if $\Sigma$ is minimal (i.e. $H \equiv 0$) and $\Sigma$ meets $\partial M$ orthogonally along $\partial \Sigma$ (i.e. $\nu \perp \partial M$). In this case, we say that $\Sigma$ is a \emph{free boundary solution}.

The study of free boundary problems for minimal surfaces was initiated by R. Courant in \cite{Courant40} and H. Lewy in \cite{Lewy51} in the 1940's. In the next few decades, these minimal surfaces with free boundaries were studied extensively by K. Goldhorn, S. Hildebrandt, W. J\"{a}ger and J. Nitsche (see, for example, \cite{Hildebrandt69} \cite{Nitsche69} \cite{Nitsche70} \cite{Goldhorn-Hildebrandt70} \cite{Jager70}  \cite{Hildebrandt-Nitsche79}), and later by J. Taylor \cite{Taylor77}, W. Meeks and S.T. Yau \cite{Meeks-Yau80}, and R.Ye \cite{Ye91}, among many others. In their approaches, they applied the direct method in the calculus of variations to the Dirichlet energy functional, and established the existence of minimizers with boundary lying on a given supporting surface. Boundary regularity results were obtained in a variety of settings. Interested readers are encouraged to consult the recent treatise \cite{Ulrich-Hildebrandt-Sauvigny10} \cite{Ulrich-Hildebrandt-Tromba10} on boundary value problems of minimal surfaces.

However, this approach cannot be used directly to answer \textbf{Question 1}. The main reason is that these existence theorems only produce area minimizers, they do not yield the existence of stationary minimal surfaces which are not area minimizing. It is not hard to see that for certain supporting surfaces, there are no non-trivial minimizers. For example, it does not furnish the existence of non-trivial stationary minimal surfaces within the region bounded by a closed convex surface in $\mathbb{R}^3$. Therefore, the direct method does not work unless the ambient space has some non-trivial topology.

To deal with the difficulty above, we need to construct unstable critical points to the area (or energy) functional. Along this direction, M. Struwe \cite{Struwe84} and A. Fraser \cite{Fraser00} applied the celebrated perturbed $\alpha$-energy of Sacks and Uhlenbeck \cite{Sacks-Uhlenbeck81} to the free boundary problem for minimal disks. They were able to produce non-trivial free boundary solutions with controlled Morse index using Ljusternik-Schnirelman theory. For instance, A. Fraser proved that (see Theorem 1 in \cite{Fraser00}) if $M$ is a smooth compact domain of a complete, homogeneously regular Riemannian three-manifold $\tilde{M}$, and the relative homotopy group $\pi_k(\tilde{M},\partial M) \neq 0$ for some $k \geq 2$, then either there exists a non-constant minimal disk $D$ in $\tilde{M}$ meeting $\partial M$ orthogonally along $\partial D$, or there exists a non-constant minimal two-sphere in $\tilde{M}$. Moreover, the Morse index of such a minimal surface is at most $k-2$. 

Despite these positive results in the existence theory, it is yet not enough to settle \textbf{Question 1} in complete generality due to the following reasons. First,  the minimal disk or sphere may not be embedded (in fact it may not even be immersed). Second, and more importantly, the free boundary solution may not be contained in the compact region $M$. The construction does not prevent the minimal surface from penetrating $\partial M$ in an unphysical way, and it is impossible to make it contained in $M$ without imposing (mean) convexity assumptions on $\partial M$ (see  \cite{Fraser02} \cite{Meeks-Yau80} \cite{Meeks-Yau82a}). This approach does not work in the non-convex situation because one do not expect $M$ to contain a disk-type free boundary solution. For example, the annular region  $M=B_2 \setminus B_1 \subset \mathbb{R}^3$, where $B_r$ denotes the ball of radius $r>0$ centered at the origin, does not contain any properly embedded minimal disk $D$ with free boundary on $\partial M$. If such a minimal disk $D$ were to exist, the boundary circle $\partial D$ lies on one of the boundary spheres $\partial B_1$ or $\partial B_2$. It cannot lie on $\partial B_1$ by the convex hull property (see Proposition 1.9 in \cite{Colding-Minicozzi11}). Therefore, $\partial D \subset \partial B_2$, and hence $D$ is a free boundary solution to the ball $B_2$. A uniqueness theorem of J. Nitsche \cite{Nitsche85} implies that $D$ must be a totally geodesic equatorial disk, which has non-empty intersection with the unit ball $B_1$, contradicting the hypothesis that $D \subset B_2 \setminus B_1$. On the other hand, the restriction of the equatorial disk to $B_2 \setminus B_1$ gives an example of a free boundary solution which is topologically an annulus (genus zero with two boundary components).

There is a completely different approach, using geometric measure theory, which has been very successful in constructing embedded minimal surfaces. By minimizing among all embedded disks with prescribed boundary, F. Almgren and L. Simon \cite{Almgren-Simon79} showed that any \emph{extremal curve} $\Gamma$ in $\mathbb{R}^3$, i.e. $\Gamma$ is contained in the boundary of a strictly convex domain $A \subset \mathbb{R}^3$, bounds an embedded minimal disk contained in $A$. Based on Almgren and Simon's paper, W. Meeks, L. Simon and S.T. Yau \cite{Meeks-Simon-Yau82} proved, under suitable hypothesis, the existence of an embedded minimal surface which minimizes area in its isotopy class in a Riemannian 3-manifold. This result has profound applications in 3-manifold topology.  

In a remarkable work of F. Almgren \cite{Almgren65} and J. Pitts \cite{Pitts81}, a minimax argument was used to prove that any closed Riemannian 3-manifold contains a smooth, embedded, closed minimal surface. The proof of interior regularity for such minimal surfaces was based on Schoen, Simon and Yau's curvature estimates for stable minimal surfaces \cite{Schoen-Simon-Yau75}. Using a clever curve lifting argument, L. Simon and F. Smith \cite{Smith82} were able to control the topology of the minimal surface. As a corollary, they proved that there exists an embedded minimal two-sphere in the three-sphere with arbitrary Riemannian metric. 

Adapting these ideas to the free boundary case, M. Gr\"{u}ter and J. Jost \cite{Gruter-Jost86a} proved the existence of an embedded minimal disk as a free boundary solution in any bounded, strictly convex domain in $\mathbb{R}^3$. In another paper \cite{Jost86a}, J. Jost claimed similar results hold under weaker convexity assumptions on the boundary. Unfortunately, the author was unable to verify some of the arguments in the paper. On the other hand, a partially free boundary problem was also studied by J. Jost in \cite{Jost86}. All of these results depend on certain curvature assumptions on the boundary of the ambient manifold.

In this paper, we settle \textbf{Question 1} in complete generality, without any curvature assumptions on $M$ or $\partial M$.

\begin{theorem}
For any compact Riemannian three-manifold $(M^3,g)$ with boundary $\partial M$, at least one of the following holds: 
\begin{itemize}
\item[(i)] there exists a properly embedded minimal surface $\Sigma \subset M$ with boundary $\partial \Sigma \subset \partial M$ such that $\Sigma$ meets $\partial M$ orthogonally along $\partial \Sigma$; or 
\item[(ii)] there is a closed, embedded minimal surface $\Sigma$ contained in the interior of $M$. 
\end{itemize}
Moreover, if $M$ is assumed to be smooth up to the boundary, then the minimal surface $\Sigma$ is smooth (up to the boundary in case (i)).
\end{theorem}

The proof of Theorem 1.1 employs a minimax construction similar to the one by F. Almgren \cite{Almgren65} and J. Pitts \cite{Pitts81}. Recently, T. Colding and C. De Lellis \cite{Colding-DeLellis03} wrote a detailed account of the minimax construction, and they were able to simplify some of the proofs significantly. In another recent paper \cite{DeLellis-Pellandini10}, C. De Lellis and F. Pellandini obtained a genus bound for minimal surfaces constructed by the minimax method in \cite{Colding-DeLellis03}. (Their bound is slightly weaker than the one conjectured by J. Pitts and H. Rubinstein \cite{Pitts-Rubinstein86}.) We observe that their result can be used to control the topology of the free boundary solution constructed in Theorem 1.1.

\begin{theorem}
Let $(M^3,g)$ be a Riemannian 3-manifold with boundary. Suppose the filling genus (see Definition 9.1) of $M$ is equal to $h$. Then, the minimal surface $\Sigma$ in Theorem 1.1 can be chosen such that 
\begin{itemize}
\item[(i)] if $\Sigma$ is orientable, then genus$(\Sigma) \leq h$;
\item[(ii)] if $\Sigma$ is non-orientable, then genus$(\Sigma) \leq 2h+1$.
\end{itemize}
\end{theorem} 

It is clear from the definition that the filling genus of any compact domain in $\mathbb{R}^3$ is zero. As a consequence, we have the following corollary (note that there is no closed minimal surface in $\mathbb{R}^3$).

\begin{corollary}
For any compact domain $M \subset \mathbb{R}^3$, there exists a properly embedded minimal surface $\Sigma \subset M$ with non-empty free boundary $\partial \Sigma \subset \partial M$ such that either
\begin{itemize}
\item[(i)] $\Sigma$ is orientable with genus zero (i.e. a disk with holes); or
\item[(ii)] $\Sigma$ is a non-orientable surface with genus one (i.e. a M\"{o}bius strip with holes). 
\end{itemize}
In case $M$ is diffeomorphic to the unit 3-ball, case (ii) does not happen and $\Sigma$ must be orientable.
\end{corollary}

The outline of this paper is as follows. In section 2, we describe an example which illustrates why the proof in \cite{Colding-DeLellis03} does not directly generalize to cover the free boundary problem. This example also serves as a model case which motivates many of the technical arguments in this paper. Section 3 to 8 comprise of the proof of our main result, Theorem 1.1. In section 3, we give some definitions and preliminary results. Moreover, we define two important concepts: \emph{freely stationary} and \emph{outer almost minimizing property}, which will play important roles in this paper. In section 4, we describe the min-max construction and give an outline of the proof of Theorem 1.1. In section 5 and 6, we establish the existence of varifolds which are both freely stationary and outer almost minimizing. In section 7, we study a minimization problem with partially free boundary, which is then used in section 8 to prove the boundary regularity of freely stationary varifolds satisfying the outer almost minimizing property. In section 9, the genus bound in Theorem 1.2 is proved using the result in \cite{DeLellis-Pellandini10}.

\textit{Acknowledgement.} The results in this paper form part of the author's Doctoral Dissertation \cite{Li11} at Stanford University. The author would like to express his sincere gratitude to his advisor Professor Richard Schoen for suggesting this interesting problem and providing lots of useful advice and support throughout the progress of this work. The author would also like to thank Professor Brian White for many helpful comments and discussions. The author is indebted to Professor Leon Simon, from whom he learned a lot of geometric measure theory used in this paper.


\section{An example}

In this section, we discuss the main difficulties in the proof of Theorem 1.1. We also give an example which illustrates the need for some technical arguments in this paper. 

Our proof of Theorem 1.1 is a modified version of the minimax construction described in \cite{Colding-DeLellis03}. Let us first briefly recall the minimax construction for closed (i.e. compact without boundary) manifolds. Consider the standard 3-sphere $\mathbb{S}^3$ (with the round metric, for example), and a continuous sweepout $\Gamma=\{\Sigma_t\}_{t \in [0,1]}$ by 2-spheres which degenerate to a point at $t=0$ and $1$. Our goal is to minimize the area of the maximal slice by deforming the sweepout using ambient isotopies (which could depend continuously on $t$). Suppose a minimizing sweepout exists, we expect a maximal slice in this sweepout to be a minimal surface (see Fig. 2 in \cite{Colding-DeLellis03}). In practice, the minimum may not be achieved by any sweepout. Therefore, we have to take a minimizing sequence $\{\Sigma^n_t\}_n$ of sweepouts, and a min-max sequence of surfaces $\{\Sigma^n_{t_n}\}_n$ which, after passing to a subsequence, converges in some weak sense (as varifolds) to an embedded minimal surface (possibly with multiplicity). There are two important points in this construction. First, we need to ensure that almost maximal slices are almost stationary (see Fig. 4 in \cite{Colding-DeLellis03}). This can be achieved by a ``tightening'' process (section 4 of \cite{Colding-DeLellis03}). Second, we want to choose our min-max sequence carefully so that it satisfies the \emph{almost minimizing property} which enables us to prove regularity for the limiting surface (sections 5-7 of \cite{Colding-DeLellis03}).

A natural attempt to generalize the above method to the free boundary problem is to sweep out the compact manifold $M$ by surfaces with boundary lying on $\partial M$, and we use isotopies which preserve $M$, but not necessarily pointwise fixing points on $\partial M$, to deform our sweepouts. Then, we carry out the same minimax construction as before and hope that a min-max sequence would converge to an embedded minimal surface with free boundary on $\partial M$. In fact, the limit would be stationary with respect to isotopies preserving $M$. Unfortunately, this is insufficient to conclude that the limit is a smooth free boundary solution. The main difficulty is that it may not be properly embedded if we do not have any convexity assumptions on the boundary $\partial M$. Unlike many constructions of minimal surfaces, we do not have barriers to prevent the interior of our minimal surface from touching the boundary. Let us further illustrate this point by the following example.

Let $B_r(a) \subset \mathbb{R}^3$ be the closed Euclidean 3-ball of radius $r$ centered at $a$. Suppose $M=B_1(0,0,0) \setminus B_{\frac{1}{4}}(0,0,\frac{1}{4})$. Consider the equatorial disk $\Sigma=M \cap \{x_3=0\}$, it is minimal and has free boundary on the outer boundary. However, this is not a legitimate free boundary solution because it is not properly embedded. The origin, which lies in the interior of the equatorial disk, is a point on the inner boundary of $M$. This happens since the inner boundary is not mean convex with respect to the inner unit normal (with respect to $M$). Nevertheless, the equatorial disk is a critical point of the area functional with respect to all variations preserving $M$. Not only is it stationary, but it is also \emph{$\epsilon$-almost minimizing} (see Definition 3.2 in \cite{Colding-DeLellis03}) on any sufficiently small ball for all $\epsilon >0$ with respect to these variations. This example shows that a smooth minimal surface which is stationary and almost minimizing with respect to the isotopies preserving $M$ may fail to be properly embedded, hence needs not be a free boundary solution (Figure 1). In the next section, we will discuss how to get around this problem.

\begin{figure}
\centering
\includegraphics[height=5cm]{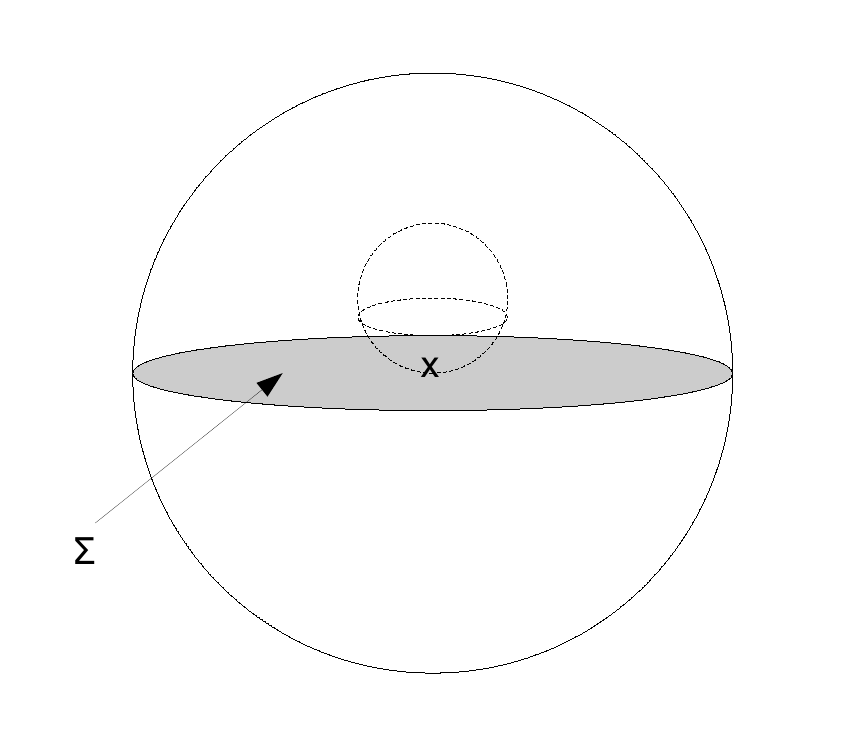}
\caption{An example of a smooth minimal surface which is stationary and almost minimizing with respect to isotopies preserving M, but not properly embedded.}
\end{figure}


\section{Definitions and preliminaries}

Let $(M^3,g)$ be a compact Riemannian 3-manifold with non-empty boundary $\partial M \neq \phi$. Suppose $M$ is connected (but $\partial M$ is not necessary connected, i.e. $M$ could have multiple boundary components). Without loss of generality, we assume that  $M$ is isometrically embedded as a compact subset of a closed Riemannian 3-manifold $\tilde{M}^3$. (Note that such an isometric embedding always exists. For example, we can smoothly extend $M$ with the metric across the boundary $\partial M$ to get a collar neighborhood which can be made cylindrical near the boundary by a cutoff function, then we take another copy of this collar neighborhood and glue the two together along the cylindrical necks.) All surfaces, with or without boundary, are smoothly embedded in $\tilde{M}$ unless otherwise stated. We will use int$(M)=M \setminus \partial M$ to denote the interior of $M$.

\subsection{Isotopies and vector fields}

We now describe the class of ambient isotopies in $\tilde{M}$ used in deforming our surfaces. An \emph{isotopy on} $\tilde{M}$ is a smooth one-parameter family  $\{\varphi_s\}_{s \in [0,1]}$ of diffeomorphisms of $\tilde{M}$, where $\varphi_0$ is the identity map of $\tilde{M}$. (The smoothness assumption here means that the map $\varphi(s,x)=\varphi_s(x):[0,1] \times \tilde{M} \to \tilde{M}$ is smooth.) Let $\mathfrak{Is}$ denote the space of all isotopies on $\tilde{M}$. Moreover, we say that an isotopy $\{\varphi_s\} \in \mathfrak{Is}$ is \emph{supported in an open set} $U \subset \tilde{M}$ if $\varphi_s(x)=x$ for every $s \in [0,1]$ and $x \in \tilde{M} \setminus U$. Define
\begin{equation*}
\mathfrak{Is}_{\text{out}}= \{\{\varphi_s\} \in \mathfrak{Is} : M \subset \varphi_s(M) \text{ for all }s \in [0,1]\}
\end{equation*}
to be the \emph{isotopies in $\tilde{M}$ which can move points out of the compact set $M \subset \tilde{M}$, but not into $M$}, and $\mathfrak{Is}_{\text{out}}(U)$ to be those in $\mathfrak{Is}_{\text{out}}$ which are supported in some open set $U \subset \tilde{M}$. Furthermore, we are also interested in situations where the compact set $M$ is preserved by the isotopy, Similarly, we define
\begin{equation*} 
\mathfrak{Is}_{\text{tan}}= \{\{\varphi_s\} \in \mathfrak{Is} : M=\varphi_s(M) \text{ for all }s \in [0,1]\}.
\end{equation*}
to be the isotopies preserving the compact set $M$, and $\mathfrak{Is}_{\text{tan}}(U)$ to be those in $\mathfrak{Is}_{\text{tan}}$ supported in some open set $U \subset \tilde{M}$. Notice that $\mathfrak{Is}_{\text{tan}}(U) \subset \mathfrak{Is}_{\text{out}}(U)$ for any open set $U \subset \tilde{M}$.

One way to generate isotopies is to consider the flow of a vector field. Let $\chi$ be the vector space of smooth vector fields on $\tilde{M}$. We define two subspaces $\chi_{\text{out}}$ and $\chi_{\text{tan}}$ of $\chi$ which correspond to the two classes of isotopies defined above. More precisely, we let 
\begin{equation*}
\chi_{\text{out}}=\{X \in \chi : X(x) \cdot \nu(x) \geq 0 \text{ for every }x \in \partial M\}
\end{equation*}
where $\nu$ is the unit outward normal of $\partial M$ with respect to $M$ and 
\begin{equation*} 
\chi_{\text{tan}}=\{X \in \chi : X(x) \cdot \nu(x) = 0 \text{ for every } x \in \partial M \}.
\end{equation*}
Notice again that $\chi_{\text{tan}} \subset \chi_{\text{out}}$. Each $X \in \chi$ generates a unique isotopy $\{ \varphi_s\}_{s\in[0,1]}$ on $\tilde{M}$ by its flow (since $\tilde{M}$ is closed, the flow exists for all time). Clearly, if $X \in \chi_{\text{out}}$, then $\{\varphi_s\} \in \mathfrak{Is}_{\text{out}}$; if $X \in \chi_{\text{tan}}$, then $\{\varphi_s\} \in \mathfrak{Is}_{\text{tan}}$.

\subsection{Varifolds and restrictions}

Varifolds are fundamental in any min-max construction, compared to other generalized surfaces like currents, because they do not allow for cancellation of mass (see p.24 in \cite{Pitts81}). We will discuss some less standard facts about varifolds. For a more comprehensive treatment, one can refer to \cite{Allard72} \cite{Colding-DeLellis03} \cite{Lin-Yang02} \cite{Simon83}.

Let $\mathcal{V}(\tilde{M})$ denote the space of 2-varifolds on $\tilde{M}$ endowed with the weak topology (see \cite{Simon83}), and $\mathcal{V}(M) \subset \mathcal{V}(\tilde{M})$ be the subspace of 2-varifolds supported in $M$. There is a restriction map 
\begin{equation*}
(\cdot) \llcorner_M: \mathcal{V}(\tilde{M}) \to \mathcal{V}(M) \subset \mathcal{V}(\tilde{M})
\end{equation*} 
defined by $V\llcorner_M(B)=V(B \cap G(M))$ for any $B \subset G(\tilde{M})$, where $G(M)$ and $G(\tilde{M})$ denote the 2-Grassmannian over $M$ and $\tilde{M}$ respectively. Since $M$ is compact, by 2.6.2 (c) in \cite{Allard72}, the restriction map is only upper semi-continuous in the weak topology in the following sense: if $V_i$ is a sequence in $\mathcal{V}(\tilde{M})$ converging weakly to $V$, then $\limsup_{i \to \infty} \|V_i\|(M) \leq \|V\|(M)$. The following lemma shows that if we have equality $\lim_{i \to \infty} \|V_i\|(M) = \|V\|(M)$, then $V_i \llcorner_M$ converges weakly to $V \llcorner M$.

\begin{lemma}
Let $V \in \mathcal{V}(\tilde{M})$. Suppose $V_i \in \mathcal{V}(\tilde{M})$ is a sequence of varifolds converging weakly to $V$ as $i \to \infty$. If the masses $\|V_i\|(M)$ converges to $\|V\|(M)$ as $i \to \infty$, then the restricted varifolds $V_i \llcorner_M$ converges weakly to $V \llcorner_M$ as $i \to \infty$.
\end{lemma}

The proof of Lemma 3.1 is rather elementary and will be given in Appendix A for the sake of completeness.

\subsection{Freely stationary varifolds}

Let $V \in \mathcal{V}(\tilde{M})$. If we take a vector field $X \in \chi_{\text{tan}}$ and let $\{\varphi_s\}_{s \in (-\epsilon, \epsilon)}$ be the flow generated by $X$ and $(\varphi_s)_\sharp V$ be the pushforward of $V$ by $\varphi_s$, then $\|(\varphi_s)_\sharp V\|(M)$ is a smooth function in $s$ since $\varphi_s(M)=M$ for all $s$. Differentiating with respect to $s$ at $s=0$, the same calculation as that in the standard first variation formula (see \cite{Simon83} for example) shows that 
\begin{equation}
\delta_M V(X) =\left. \frac{d}{ds} \right|_{s=0} \|(\varphi_s)_\sharp V \|(M) = \int_{(x,S) \in G(M)} \text{div}_S \; X(x) \; dV(x,S).
\end{equation} 
Notice that we are only integrating over $G(M)$ instead of $G(\tilde{M})$ on the right hand side of (3.1) since we are only counting area in $M$.

\begin{definition}
A varifold $V \in \mathcal{V}(\tilde{M})$ is said to be \emph{freely stationary in an open set $U \subset \tilde{M}$} if and only if $\delta_M V (X)=0$ for all vector fields $X \in \chi_{\text{tan}}$ supported in $U$. If $U=\tilde{M}$, we simply say that $V$ is \emph{freely stationary}.
\end{definition}

We denote the set of varifolds which are freely stationary in $U$ by $\mathcal{V}_{\infty,U} \subset \mathcal{V}(\tilde{M})$ and the set of freely stationary varifolds by $\mathcal{V}_\infty$.

\begin{remark}
It is obvious that $\delta_M V (X) = \delta_M V\llcorner_M (X)$ for any $V \in \mathcal{V}(\tilde{M})$ and $X \in \chi_{\text{tan}}$. Therefore, if $V \in \mathcal{V}(\tilde{M})$ is freely stationary in $U$, then so is $V \llcorner_M$ and vice versa. So, we often assume that freely stationary varifolds are supported in $M$.
\end{remark}

Using the compactness of mass bounded varifolds in the weak topology and (3.1), it is immediate that the set of mass bounded freely stationary varifolds supported in $M$ is compact in the weak topology.

\begin{lemma}
For any open set $U \subset \tilde{M}$, and any constant $C>0$, the set 
\begin{equation*}
\mathcal{V}^C_{\infty,U}(M) = \{ V \in \mathcal{V}_{\infty,U}(M) : \|V\|(M) \leq C\}  
\end{equation*}
is compact in the weak topology.
\end{lemma}

\begin{proof}
Let $V_i$ be a sequence of varifolds in $\mathcal{V}^C_{\infty,U}(M) $. Since $V_i$ are supported in $M$, $\|V_i\|(\tilde{M})=\|V_i\|(M)$ are uniformly bounded by $C$. Therefore, a subsequence of $V_i$ (after relabeling) converges weakly to $V$ in $\mathcal{V}(\tilde{M})$ by compactness of mass bounded varifolds. Since $\mathcal{V}(M)$ is a closed subset of $\mathcal{V}(\tilde{M})$, the limit $V$ is also supported in $M$, i.e. $V \in \mathcal{V}(M)$. Therefore, $\|V\|(M)=\|V\|(\tilde{M})=\lim \|V_i\|(\tilde{M})=\lim \|V_i\|(M) \leq C$. It remains to show that $V$ is freely stationary in $U$, but this follows directly from the first variation formula (3.1) and the fact that $V_i$ and $V$ are supported in $M$.
\end{proof}

In this paper, we will use some results in \cite{Gruter-Jost86}, where the monotonicity formula and the Allard regularity for freely stationary varifolds were proved. Their proofs were given for rectifiable varifolds in $\mathbb{R}^N$ but they can be easily generalized to general varifolds in Riemannian manifolds. As a result, we have the following monotonicity formula for any freely stationary varifold $V$ in $\tilde{M}$: there exists a constant $R>0$ (depending only on the geometry of $M$ and $\partial M$), and a function $C(r) \geq 1$ such that for any $x \in \partial M$ and $0 \leq \sigma \leq \rho \leq r \leq R$,
\begin{equation}
\frac{\|V\|(M \cap B_\sigma(x))}{\sigma^2} \leq C(r) \frac{\|V\|(M \cap B_\rho(x))}{\rho^2}.
\end{equation}
Note that $C(r)$ goes to $1$ as $r \to 0$, so the density $\theta(x,V)$ of the freely stationary varifold $V$ at $x \in \partial M$ is well-defined.

In \cite{Gruter-Jost86a}, the well-known Schoen curvature estimates \cite{Schoen83} for stable minimal surfaces were generalized to the free boundary case. The compactness theorem for stable minimal surfaces with free boundary follows easily from the curvature estimates.

\begin{lemma}[Gr\"{u}ter-Jost \cite{Gruter-Jost86}]
Let $U \subset \tilde{M}$ be an open set. Suppose $\{\Sigma^n\}$ is a sequence of properly embedded stable minimal surfaces in $U \cap M$ with free boundary on $U \cap \partial M$, and their areas are uniformly bounded. Then, for any compact subset $K \subset \subset U$, there is a subsequence of $\Sigma^n$ converging smoothly (multiplicity allowed) to a properly embedded stable minimal surface $\Sigma^\infty$ in $K \cap M$ with free boundary lying on $K \cap \partial M$.
\end{lemma}

\subsection{Outer almost minimizing property}

In general, a stationary varifold may possess singularity and hence is not a smooth minimal surface. Even if it is smooth, the example in section 2 shows that an embedded smooth minimal surface which is freely stationary may fail to be properly embedded. In order to achieve regularity and properness of the freely stationary varifold, we require a stronger condition called \emph{outer almost minimizing property}. Roughly speaking, a surface is outer almost minimizing means that if you want to decrease its area in $M$ through an outward isotopy, its area in $M$ must become large at some time during the deformation. The precise definition is given below.

\begin{definition}
Given $\epsilon>0$ and an open set $U \subset \tilde{M}$, a varifold $V \in \mathcal{V}(\tilde{M})$ is \emph{$\epsilon$-outer almost minimizing in $U$} if and only if there does not exist isotopy $\{\varphi_s\}_{s \in [0,1]} \in \mathfrak{Is}_{\text{out}}(U)$ such that 
\begin{enumerate}
	\item $\|(\varphi_s)_\sharp V\|(M) \leq \|V\|(M) + \frac{\epsilon}{8}$ for all $s \in [0,1]$;
	\item $\|(\varphi_1)_\sharp V\|(M) \leq \|V\|(M)-\epsilon$.
\end{enumerate} 
A sequence $\{V^n\} \subset \mathcal{V}(\tilde{M})$ is said to be an \emph{outer almost minimizing sequence in $U$} if each $V^n$ is $\epsilon_n$-outer almost minimizing in $U$ for some sequence $\epsilon_n \downarrow 0$. Moreover, if $V^n$ converges weakly to some $V$ in $U$, then we say that $V$ is \emph{outer almost minimizing in $U$}.
\end{definition}

\begin{remark}
The definition is almost the same as the one used in \cite{Colding-DeLellis03} except that we are considering the area in $M$ and outward isotopies only.
\end{remark}

It is not hard to see that any outer almost minimizing varifold is freely stationary.

\begin{proposition}
If $V$ is outer almost minimizing in $U$, then $V$ is freely stationary in $U$.
\end{proposition}

\begin{proof}
Since $V$ is outer almost minimizing in $U$, there exists, by definition, a sequence $\epsilon_n \downarrow 0$ for which $V^n$ is $\epsilon_n$-outer almost minimizing in $U$. We will prove the proposition by contradiction.

Suppose, in contrary, that $V$ is not freely stationary in $U$. Then there is a vector field $X \in \chi_{\text{tan}}$ supported in $U$ such that 
\begin{equation*}
\delta_M V(X) \leq -c <0
\end{equation*}
for some real constant $c>0$. Let $\{\varphi_s\}_{s \in \mathbb{R}}$ be the flow generated by $X$. Suppose $V_s$ and $V^n_s$ are the pushforwards of the varifolds $V$ and $V^n$, respectively, by the diffeomorphism $\varphi_s$. Since $X \in \chi_{\text{tan}}$, (3.1) implies that $\delta_M V_s(X)$ is a continuous function in $s$ (for $X$ fixed). Therefore, we have 
\begin{equation}
\delta_M V_s(X) \leq -\frac{c}{2} < 0
\end{equation}
for all $s \in [0,s_0]$ for some $s_0>0$. We claim that for all sufficiently large $n$, it holds true that
\begin{equation}
\delta_M V^n_s(X) \leq -\frac{c}{4} < 0
\end{equation}
for all $s \in [0,s_0]$.

Let us assume (3.4) for the moment, integrating in $s$ gives
\begin{equation}
\|V^n_s\|(M) \leq \|V^n\|(M) - \frac{cs}{4}
\end{equation}
for all $s \in [0,s_0]$. Since $V^n$ is $\epsilon_n$-outer almost minimizing in $U$, this implies
\begin{equation}
\|V^n_{s_0}\|(M) \geq \|V^n\|(M)-\epsilon_n
\end{equation}
for all $n$ sufficiently large. Combining (3.5) and (3.6), we have $\epsilon_n \geq \frac{cs_0}{4}>0$ for $n$ sufficiently large. This is a contradiction since $\epsilon_n \downarrow 0$.

It remains to verify (3.4). Let $f_n(s)=\delta_M V^n_s(X)$ and $f(s)=\delta_M V_s(X)$. From the definition of pushforward of a varifold and (3.1), $\{f_n\}_{n \in \mathbb{N}}$ is an equicontinuous family of functions on $[0,s_0]$. Furthermore, 
\begin{equation}
\limsup_{n \to \infty} f_n(s) \leq f(s)
\end{equation} 
for all $s \in [0,s_0]$. Our goal is to prove that for any $\epsilon>0$, $f_n(s) \leq f(s)+\epsilon$ for all $n$ sufficiently large and all $s \in [0,s_0]$. If not, then there exists a sequence $n_k \to \infty$ and $s_k \in [0,s_0]$ such that $f_{n_k}(s_k) > f(s_k)+\epsilon$. Without loss of generality, we can assume $s_k \to  s_\infty$ for some $s_\infty \in [0,s_0]$.  By equicontinuity, for $k$ sufficiently large, we have 
\begin{equation*}
f_{n_k}(s_\infty)+ \frac{\epsilon}{2} \geq f_{n_k}(s_k) > f(s_k)+\epsilon.
\end{equation*} 
Take $k \to \infty$, we get $\limsup_{n \to \infty} f_n(s_\infty) \geq \limsup_{k \to \infty} f_{n_k}(s_\infty) \geq f(s_\infty)+\epsilon/2$, which contradicts (3.7). Finally, (3.4) follows from (3.3) by taking $\epsilon=c/4$.

\end{proof}

The example in section 2 shows that the converse of Proposition 3.8 is false (Figure 1). The equatorial disk shown in Figure 1 is freely stationary and almost minimizing with respect to isotopies preserving $M$, but it fails to be outer almost minimizing. By allowing more deformations, we can rule out such cases as in Figure 1 and obtain a properly embedded free boundary solution from a min-max construction to be described in section 4.  


\section{The min-max construction}

In this section, we describe a min-max construction for properly embedded minimal surfaces with free boundary in any compact Riemannian $3$-manifold with boundary.

\subsection{Sweepouts}

First, we define a generalized family of surfaces which allow mild singularities and changes in topology. We will always parametrize a sweepout by the letter $t$ over the interval $[0,1]$ unless otherwise stated.

\begin{definition}
A family $\{ \Sigma_t\}_{t \in [0,1]}$  of surfaces in $\tilde{M}$ is said to be a \emph{generalized smooth family of surfaces}, or simply a \emph{sweepout}, if and only if there exists a finite subset $T \subset [0,1]$ and a finite set of points $P \subset \tilde{M}$ such that 
\begin{itemize}
	\item[(1)] for $t \notin T$, $\Sigma_t$ is a smoothly embedded closed surface (not necessarily connected) in $\tilde{M}$;
	\item[(2)] for $t \in T$, $\Sigma_t \setminus P$ is a smoothly embedded surface (not necessarily connected) in $\tilde{M}$ and $\Sigma_t$ is compact; and
	\item[(3)] $\Sigma_t$ varies smoothly in $t$ (see Remark 4.2 below).
\end{itemize}
If, in addition to (1)-(3) above,
\begin{itemize}
	\item[(4)] $\mathcal{H}^2(\Sigma_t \cap M)$ is a continuous function in $t \in [0,1]$, here $\mathcal{H}^2$ is the 2-dimensional Hausdorff measure induced by the metric on $\tilde{M}$,
\end{itemize}
we say that $\{\Sigma_t\}_{t \in [0,1]}$ is a \emph{continuous sweepout}.
\end{definition}

\begin{remark}
The smoothness condition in (3) means the following: for each $t \notin T$, for $\tau$ close enough to $t$, $\Sigma_\tau$ is a graph over $\Sigma_t$ (hence diffeomorphic to $\Sigma_t$) and $\Sigma_\tau$ converges smoothly to $\Sigma_t$ as a graph when $\tau \to t$. At $t \in T$, for any $\epsilon>0$ small, let $P_\epsilon=\{x \in \tilde{M}:d(x,P)<\epsilon\}$, then $\Sigma_\tau \setminus P_\epsilon$ converges smoothly to $\Sigma_t \setminus P_\epsilon$ in the graphical sense above as $\tau \to t$.
\end{remark}

\begin{remark}
Note that condition (4) is not redundant since we could have a continuous (or even smooth) family of $\Sigma_t \subset \tilde{M}$ such that $\mathcal{H}^2(\Sigma_t \cap M)$ is discontinuous. In general, the function $\mathcal{H}^2(\Sigma_t \cap M)$ is only upper semi-continuous in $t$. See Figure 4.1 for an example of a discontinuous sweepout of a topological annulus by curves. Here we let $M$ be a disk with a square removed and consider a sweepout of $M$ by vertical lines.This example is one dimension lower but a similar example in $\mathbb{R}^3$ can be constructed easily. In fact, by Lemma 3.1, condition (4) is equivalent to saying that $\{\Sigma_t \cap M\}_{t \in [0,1]}$ is a continuous family as varifolds in $M$.
\end{remark}

\begin{figure}
\centering
\includegraphics[height=8cm]{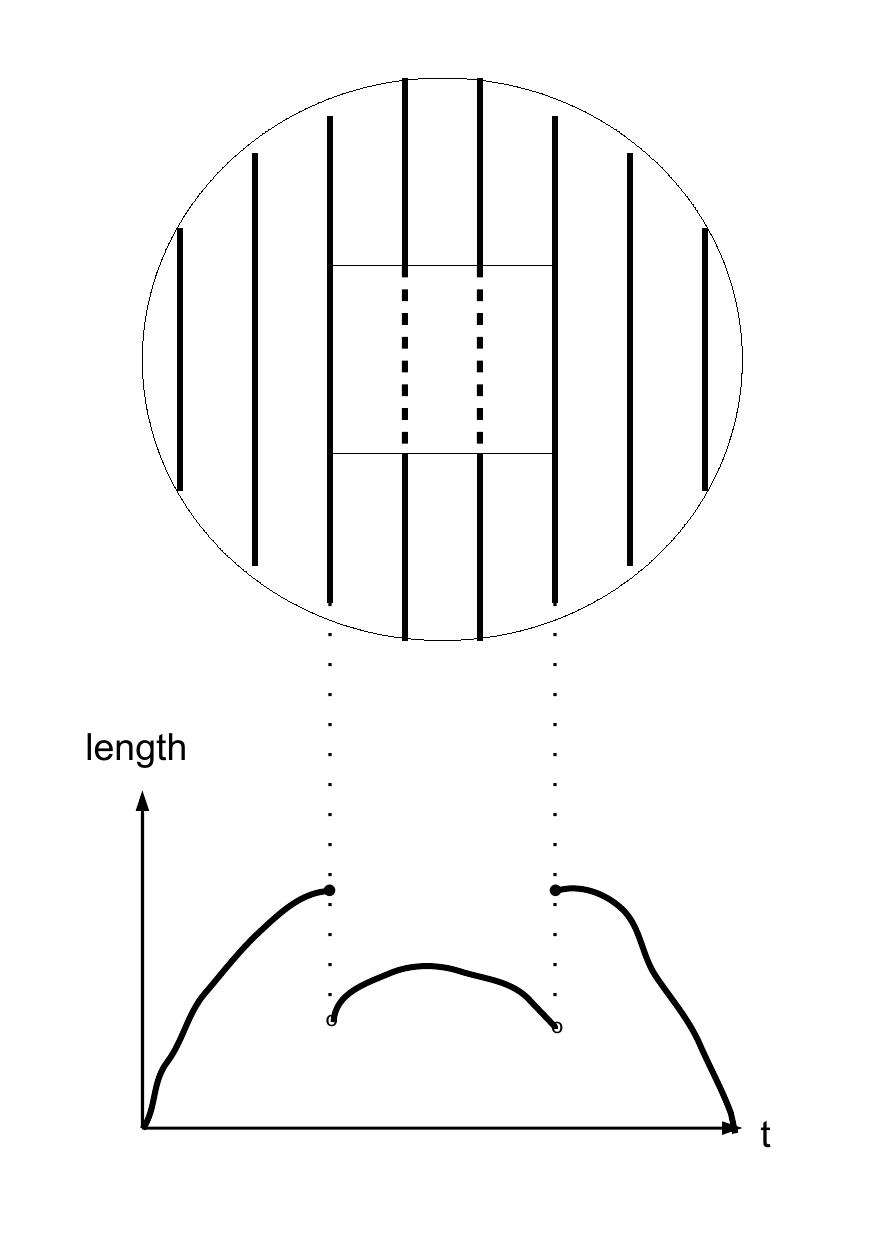}
\caption{A sweepout which is not continuous. (The dashed lines represent the part of the line which lies outside $M$, whose length is not counted.)}
\end{figure}

\subsection{The min-max construction}

Given a sweepout $\{\Sigma_t\}$, we can deform the sweepout to get another sweepout by the following procedure. Let $\psi=\psi_t(x)=\psi(t,x): [0,1] \times \tilde{M} \to \tilde{M}$ be a smooth map such that for each $t \in [0,1]$, there exists isotopies $\{\varphi^t_s\}_{s \in [0,1]} \in \mathfrak{Is}^{\text{out}}$ such that $\varphi^t_1=\psi_t$. We define a new family $\{ \Sigma'_t\}$ by $\Sigma'_t=\psi_t(\Sigma_t)$. It is clear that $\{\Sigma'_t\}$ is a sweepout in the sense of Definition 4.1. A collection $\Lambda$ of sweepouts is \emph{saturated} if it is closed under these deformations of sweepouts. 

\begin{remark}
For technical reasons, we will assume that any saturated collection $\Lambda$ of sweepouts has the additional property that there exists some natural number $N=N(\Lambda) < \infty$ such that for any $\{\Sigma_t\} \in \Lambda$, the set $P$ in Definition 4.1 consists of at most $N$ points.
\end{remark}

We will apply our min-max construction to a saturated collection of sweepouts. Given any such collection $\Lambda$, and any sweepout $\{\Sigma_t\} \in \Lambda$, we denote by $\mathcal{F}(\{\Sigma_t\})$ the area of its maximal slice (with respect to area in $M$) and by $m_0(\Lambda)$ the infimum of $\mathcal{F}$ over all sweepouts in $\Lambda$; that is,
\begin{equation*} 
\mathcal{F}(\{\Sigma_t\})=\sup_{t \in [0,1]} \mathcal{H}^2(\Sigma_t \cap M),
\end{equation*}
and
\begin{equation*}
m_0(\Lambda)=\inf_{\{\Sigma_t\} \in \Lambda} \mathcal{F}(\{\Sigma_t\}).
\end{equation*}
Note that we have to take ``sup'' in the definition of $\mathcal{F}$ instead of ``max'' (as in \cite{Colding-DeLellis03}) because the maximum may not be achieved if the sweepout is not continuous.

\begin{definition}
Given a saturated collection $\Lambda$ of sweepouts, 
\begin{itemize}
	\item[(1)] A sequence $\{\Sigma_t^n\}_{n \in \mathbb{N}}$ of sweepouts in $\Lambda$ is a \emph{minimizing sequence of sweepouts} if 
	\begin{equation*}
	\lim_{n \to \infty} \mathcal{F}(\{\Sigma^n_t\})=m_0(\Lambda).
	\end{equation*}
	\item[(2)] Let $\{\Sigma_t^n\}_{n \in \mathbb{N}}$ be a minimizing sequence of sweepouts. Suppose we have a sequence $t_n \in [0,1]$. We say that $\{\Sigma^n_{t_n} \cap M\}_{n \in \mathbb{N}}$ is a \emph{min-max sequence of surfaces} if
	\begin{equation*}
	\lim_{n \to \infty} \mathcal{H}^2(\Sigma^n_{t_n} \cap M)=m_0(\Lambda).
	\end{equation*}
\end{itemize}
\end{definition}

Our goal is to show that there exists some min-max sequence $\Sigma^n_{t_n} \cap M$ converging (in the varifold sense) to a properly embedded free boundary solution $\Sigma$ (possibly with multiplicities). It is clear that the area of $\Sigma$ (counting multiplicities) is equal to $m_0(\Lambda)$. In order to produce something non-trivial, we need  $m_0(\Lambda) > 0$. We first show by an isoperimetric inequality that this can be done by choosing an initial sweepout to be the level sets of a Morse function.

\begin{proposition}
There exists a saturated collection $\Lambda$ of sweepouts with $m_0(\Lambda)>0$.
\end{proposition}

\begin{proof}
Take any Morse function $\tilde{f}:\tilde{M} \to [0,1]$ on the closed 3-manifold $\tilde{M}$. Define $\Sigma_t=\tilde{f}^{-1}(t)$ for $t \in [0,1]$. Then $\{\Sigma_t\}_{t \in [0,1]}$ is a sweepout in the sense of Definition 4.1. Let $\Lambda$ be the \emph{saturation} of $\{\Sigma_t\}$, the smallest collection of sweepouts which is saturated and contains $\{\Sigma_t\}$. We will show that for such a collection $\Lambda$, we have $m_0(\Lambda)>0$.

Let $\psi=\psi_t(x)=\psi(t,x): [0,1] \times \tilde{M} \to \tilde{M}$ be a smooth map such that for each $t \in [0,1]$, there exists isotopies $\{\varphi^t_s\}_{s \in [0,1]} \in \mathfrak{Is}_{\text{out}}$ such that $\varphi^t_1=\psi_t$. Define the new sweepout $\{\Gamma_t\} \in \Lambda$ by $\Gamma_t=\psi_t(\Sigma_t)$. We claim that $\mathcal{F}(\{\Gamma_t\}) \geq C >0$, where $C$ is a constant independent of $\psi$. This would imply $m_0(\Lambda) \geq C >0$.

To prove our claim, let $U_t=\tilde{f}^{-1}([0,t))$, $U'_t=\tilde{M} \setminus U_t$ and take $V_t=\psi_t(U_t)$, $V'_t=\psi_t(U'_t)$. The compact subset $M$ is a disjoint union of $V_t \cap M$ and $V'_t \cap M$, with $\partial V_t \cap $ int$(M)=\partial V'_t \cap $ int$(M)$. Since the function $t \mapsto \mathcal{H}^3(V_t \cap M)$ is continuous, and $\mathcal{H}^3(V_0 \cap M)=0$, $\mathcal{H}^3(V_1 \cap M)=$ Vol$(M)$, there exists $t_0 \in (0,1)$ such that $\mathcal{H}^3(V_{t_0} \cap M)=\frac{1}{2} \text{Vol}(M)$.

By the isoperimetric inequality, there exists a constant $C=C(M)>0$ such that 
\begin{equation*} 
\frac{1}{2} \text{Vol}(M)=\mathcal{H}^3(V_{t_0} \cap M) \leq C(M)(\mathcal{H}^2(\Gamma_{t_0} \cap M))^{\frac{3}{2}}. 
\end{equation*}
Hence,
\begin{equation*} 
\mathcal{F}(\{\Gamma_t\})=\sup_{t \in [0,1]} \mathcal{H}^2(\Gamma_t \cap M) \geq \left( \frac{\text{Vol}(M)}{2 \; C(M)}\right)^{\frac{2}{3}}>0. 
\end{equation*}
This proves our claim and thus the proposition.
\end{proof}

\subsection{Convergence of min-max sequences}

We now state the main convergence result which implies Theorem 1.1.

\begin{theorem}
Let $M$ be a compact domain of a closed Riemannian 3-manifold $\tilde{M}$. Given any saturated collection of sweepouts $\Lambda$, there exists a min-max sequence of surfaces $\{\Sigma^n_{t_n} \cap M\}_{n \in \mathbb{N}}$ obtained from $\Lambda$, which converges in the sense of varifolds to an integer-rectifiable varifold $V$ in $M$ with $\|V\|(M)=m_0(\Lambda)$. Moreover, there exists natural numbers $n_1,\ldots,n_k$ and smooth compact properly embedded minimal surfaces $\Gamma_1,\ldots,\Gamma_k$ such that 
\begin{equation*}
V=\sum_{i=1}^k n_i \Gamma_i ,
\end{equation*} 
where each $\Gamma_i$ is either closed or meets $\partial M$ orthogonally along the free boundary $\partial \Gamma_i$. 
\end{theorem}

The proof of Theorem 4.7 can be divided into three parts. The first part is a tightening argument which is similar to Birkhoff's curve shortening process \cite{Birkhoff17}. The goal is to find a \emph{good} minimizing sequence of sweepouts such that almost maximal slices are almost freely stationary. This rules out the existence of almost maximal \emph{bad} slices (see Figure 4 in \cite{Colding-DeLellis03}) which would not converge to a freely stationary varifold. The precise statement we will prove is the following:

\begin{proposition}
Given a saturated collection $\Lambda$ of sweepouts, there exists a minimizing sequence of sweepouts $\{ \Sigma^n_t\}_{n \in \mathbb{N}}$ such that
\begin{itemize} 
\item[(1)] $\{ \Sigma^n_t\}_{t \in [0,1]}$ is a continuous sweepout for each $n \in \mathbb{N}$.
\item[(2)] Every min-max sequence of surfaces $\{\Sigma^n_{t_n} \cap M\}_{n \in \mathbb{N}}$ constructed from such a minimizing sequence has a subsequence converging weakly to a freely stationary varifold $V \in \mathcal{V}_\infty$ supported in $M$. 
\end{itemize}
\end{proposition}

The key new ingredient in the proposition above is a perturbation lemma which says that any sweepout can be approximated by a continuous sweepout.

\begin{lemma}
Given any sweepout $\{\Sigma_t\}_{t \in [0,1]} \in \Lambda$, and any $\epsilon >0$, there exists a continuous sweepout $\{\Sigma_t'\}_{t \in [0,1]} \in \Lambda$ such that 
\begin{equation*}
\mathcal{F}(\{\Sigma'_t\}) \leq \mathcal{F}(\{\Sigma_t\})+ \epsilon.
\end{equation*} 
\end{lemma}

The proof of Lemma 4.9 is rather technical and will be presented in Appendix B. Using the perturbation lemma, one can assume, without loss of generality, that a minimizing sequence of sweepouts is continuous. Once we have continuity, the proof of Proposition 4.8 is a simple modification of the arguments in section 4 of \cite{Colding-DeLellis03}. We will give the details in section 5 of this paper.

The second part of the proof of Theorem 4.7 is to establish the existence of a min-max sequence which is outer almost minimizing on small annuli. Let $x \in \tilde{M}$ and $r>0$. Let $\mathcal{A}_r(x)$ be the collection of all annulus centered at $x$ with outer radius less than $r$ and inner radius greater than zero. We will prove the following existence result in section 6. A key note is that the proof requires continuity of the minimizing sequence of sweepout, which is given by Proposition 4.8.

\begin{proposition}
Given a saturated collection $\Lambda$, there exists a positive function $r:\tilde{M} \to \mathbb{R}$ and a min-max sequence of surfaces  $\{\Sigma^n=\Sigma^n_{t_n} \cap M\}_{n \in \mathbb{N}}$ such that 
\begin{itemize}
	\item[(1)] $\{\Sigma^n\}_{n \in \mathbb{N}}$ is an outer almost minimizing sequence in any annulus An $\in \mathcal{A}_{r(x)}(x)$, where $x$ is any point in $\tilde{M}$;
	\item[(2)] In every such annulus An, $\Sigma^n$ is a smooth surface (possibly with boundary) when $n$ is sufficiently large;
	\item[(3)] The sequence $\Sigma^n$ converges to a freely stationary varifold $V$ in $M$ as $n \to \infty$.
\end{itemize}
\end{proposition}

The third part of the proof of Theorem 4.7 is a regularity theorem for outer almost minimizing varifolds. The idea is that the outer almost minimizing property enables us to construct replacements (see Definition 8.1) for the freely stationary varifold $V$ obtained in Proposition 4.10. It turns out that having sufficiently many replacements implies that a freely stationary varifold is regular.

\begin{proposition}
The freely stationary varifold $V$ in Proposition 4.10 is integer rectifiable and there exists natural numbers $n_1,\ldots,n_k$ and smooth compact properly embedded minimal surfaces $\Gamma_1,\ldots,\Gamma_k$ such that 
\begin{equation*}
V=\sum_{i=1}^k n_i \Gamma_i ,
\end{equation*} 
where each $\Gamma_i$ is either closed or meets $\partial M$ orthogonally along the free boundary $\partial \Gamma_i$. 
\end{proposition}

The construction of replacements involves a minimization problem among all surfaces which are outward isotopic to a fixed surface. This is a localized version of Meeks-Simon-Yau's paper \cite{Meeks-Simon-Yau82} with partially free boundary. In section 7, we will treat this minimization problem in detail. In section 8, we prove the regularity of freely stationary varifolds which can be replaced sufficiently many times. Combining Proposition 4.8, 4.10 and 4.11, we obtain the main convergence result (Theorem 4.7).


\section{Existence of freely stationary varifolds}

In this section, we show that there exists a nice minimizing sequence of sweepout such that \emph{any} min-max sequence of surfaces obtained from such a minimizing sequence has a subsequence converging to a varifold in $M$ which is freely stationary. Using the perturbation lemma (Lemma 4.9), we can make the minimizing sequence continuous. This is essential in the proof of Proposition 4.8 and Proposition 4.10 in the next section.

We restate Proposition 4.8 below.

\begin{proposition}[Proposition 4.8]
Given a saturated collection $\Lambda$ of sweepouts, there exists a minimizing sequence of sweepouts $\{ \Sigma^n_t\}_{n \in \mathbb{N}}$ such that
\begin{itemize} 
\item[(1)] $\{ \Sigma^n_t\}_{t \in [0,1]}$ is a continuous sweepout for each $n \in \mathbb{N}$.
\item[(2)] Every min-max sequence of surfaces $\{\Sigma^n_{t_n} \cap M\}_{n \in \mathbb{N}}$ constructed from such a minimizing sequence has a subsequence converging weakly to a freely stationary varifold $V \in \mathcal{V}_\infty$ supported in $M$. 
\end{itemize}
\end{proposition}

\begin{proof}
Let $\{\Sigma^n_t\}_{n \in \mathbb{N}} \subset \Lambda$ be a minimizing sequence of sweepouts. By Lemma 4.9, we can assume $\{\Sigma^n_t\}_{t \in [0,1]}$ is a continuous sweepout for each $n$. So (1) is established. 

Fix some $C>4m_0$. By Lemma 3.4, $\mathcal{V}^C_\infty(M) \subset \mathcal{V}^C(M)$ is a compact set in the weak topology. Let $d$ be a metric on $\mathcal{V}^C(M)$ whose metric topology agrees with the weak topology. By restricting ourselves to tangential vector fields in $\chi_{\text{tan}}$, the same argument as in section 4 of \cite{Colding-DeLellis03} gives  a ``tightening'' map 
\begin{equation*}
\Psi: \mathcal{V}^C(M) \to \mathfrak{Is}_{\text{tan}}
\end{equation*} 
such that 
\begin{itemize}
\item[(a)] $\Psi$ is continuous with respect to the weak topology on $\mathcal{V}^C(M)$ and the $C^1$-norm on $\mathfrak{Is}_{\text{tan}}$.
\item[(b)] If $V \in \mathcal{V}^C_\infty(M)$, then $\Psi(V)$ is the identity isotopy on $\tilde{M}$.
\item[(c)] If $V \notin \mathcal{V}^C_\infty(M)$, then 
\begin{equation*}
\|(\Psi(V)_1)_\sharp V\|(M) \leq \|V\|(M) - L(d(V,\mathcal{V}^C_\infty(M))
\end{equation*} 
for some continuous strictly increasing function $L:\mathbb{R} \to \mathbb{R}$ with $L(0)=0$. 
\end{itemize}

\begin{remark}
In step 1 of section 4 of \cite{Colding-DeLellis03}, we should take $\|\chi_V\|_\infty \leq 1/k$ when $k>0$ to make sure that $\Psi$ is continuous. The rest of the argument goes through because the set of tangential vector fields $\chi_{\text{tan}}$ is a convex subset of the set of all vector fields $\chi$ in $\tilde{M}$. Hence, the vector field $H_V$ defined in step 1 of section 4 of \cite{Colding-DeLellis03} also belongs to $\chi_{\text{tan}}$.
\end{remark}

Since $\{\Sigma^n_t\}$ are continuous sweepouts, for each $n$, $\{\Sigma^n_t \cap M\}_{t \in [0,1]}$ is a continuous family in $\mathcal{V}^C(M)$. Therefore, $\Psi(\Sigma_t \cap M)$ is a continuous family in $\mathfrak{Is}_{\text{tan}}$. By a smoothing argument (for example, one can take a convolution in the $t$-variable), we can make it a smooth family. We use these tangential isotopies to deform the minimizing sequence $\{\Sigma^n_t\}$ to another minimizing sequence $\{\Gamma^n_t\}$ which satisfies 
\begin{equation}
\mathcal{H}^2(\Gamma^n_t \cap M) \leq \mathcal{H}^2(\Sigma^n_t \cap M) - L(d(\Sigma^n_t \cap M, \mathcal{V}_\infty^C(M)))/2.
\end{equation}
As $\{\Sigma^n_t\}_n$ is a minimizing sequence of sweepouts, we can assume that 
\begin{equation}
\mathcal{F}(\{\Sigma^n_t\}) \leq m_0 + 1/n.
\end{equation}
Furthermore, the sweepouts $\{\Gamma^n_t\}$ are continuous since only tangential vector fields are used in the deformations. 

Next, we claim that for every $\epsilon >0$, there exist $\delta>0$ and $N \in \mathbb{N}$ such that whenever $n>N$ and $\mathcal{H}^2(\Gamma^n_{t_n} \cap M) > m_0 -\delta$, we have $d(\Gamma^n_{t_n} \cap M,\mathcal{V}_\infty^C(M)) < \epsilon$. To see this, we argue by contradiction. Note first that the construction of the tightening map yields a continuous and increasing function $\lambda:\mathbb{R}^+ \to \mathbb{R}^+$ (independent of $t$ and $n$) such that $\lambda(0)=0$ and 
\begin{equation}
d(\Sigma_t \cap M, \mathcal{V}_\infty^C(M)) \geq \lambda(d(\Gamma_t^n \cap M,\mathcal{V}_\infty^C(M))).
\end{equation}
Fix $\epsilon >0$ and choose $\delta>0$, $N \in \mathbb{N}$ such that $L(\lambda(\epsilon))/2 -\delta >1/N$. We claim that for this choice of $\delta$ and $N$, whenever $n>N$ and $\mathcal{H}^2(\Gamma^n_{t_n} \cap M) > m_0 -\delta$, we have $d(\Gamma^n_{t_n} \cap M,\mathcal{V}_\infty^C(M)) < \epsilon$. Suppose not. Then there are $n>N$ and $t_n$ such that $\mathcal{H}^2(\Gamma^n_{t_n} \cap M) > m_0 -\delta$ and $d(\Gamma_t^n \cap M,\mathcal{V}_\infty^C(M)) \geq \epsilon$. Hence, from (5.1) and (5.3) we get 
\begin{equation*}
\mathcal{H}^2(\Sigma^n_t \cap M) \geq \mathcal{H}^2(\Gamma^n_t \cap M) +\frac{L(\lambda(\epsilon))}{2} > m_0 + \frac{L(\lambda(\epsilon))}{2}-\delta > m_0+\frac{1}{N} > m_0 + \frac{1}{n}.
\end{equation*}
This contradicts (5.2). This proves our claim and the claim clearly implies (2) in Proposition 5.1. Therefore, the proof is completed.
\end{proof}


\section{Existence of outer almost minimizing varifolds}

In this section, we prove the existence of a min-max sequence which is outer almost minimizing on small annuli. The proof is a combinatorial argument first introduced by F. Almgren in \cite{Almgren65}. First we recall the following definition from \cite{Colding-DeLellis03}.

\begin{definition}
Let $\mathcal{CO}$ be the set of pairs $(U^1,U^2)$ of open sets in $\tilde{M}$ with 
\begin{equation*}
d(U^1,U^2) > 2 \min \{\text{diam}(U^1),\text{diam}(U^2)\}.
\end{equation*}
Given $(U^1,U^2) \in \mathcal{CO}$, we say that $V \in \mathcal{V}(\tilde{M})$ is \emph{$\epsilon$-outer almost minimizing in $(U^1,U^2)$} if it is $\epsilon$-outer almost minimizing in at least one of the $U^1$ or $U^2$.
\end{definition}

\begin{remark}
The significance of $\mathcal{CO}$ is that for any $(U^1,U^2)$ and $(V^1,V^2) \in \mathcal{CO}$, there are some $i,j=1,2$ with $d(U^i,V^j)>0$, hence $U^i \cap V^j =\emptyset$. 
\end{remark}

The key lemma in this section is the following.

\begin{lemma}
Let $\{\Sigma^n_t\}$ be a minimizing sequence as given in Proposition 5.1. Then there is a min-max sequence $\{\Sigma^L\}_{L \in \mathbb{N}} = \{ \Sigma^{n(L)}_{t_{n(L)}} \}_{L \in \mathbb{N}}$ such that 
\begin{equation*} 
\text{each }\Sigma^L \text{ is } \frac{1}{L}\text{-outer almost minimizing in every } (U^1,U^2) \in \mathcal{CO}. 
\end{equation*}
\end{lemma}

\begin{proof}
We will argue by contradiction. First of all, we fix a minimizing sequence $\{ \Sigma^n_t\}_{n \in \mathbb{N}} \subset \Lambda$  satisfying Proposition 5.1 and such that 
\begin{equation}
\mathcal{F}(\{\Sigma^n_t\}) < m_0 +\frac{1}{n}. 
\end{equation}

Fix $L \in \mathbb{N}$. To prove the lemma, we make the following claim.

\textit{Claim 1:} There exists $n >L$ and $t_n \in [0,1]$ such that $\Sigma^n=\Sigma^n_{t_n}$ satisfies
\begin{itemize} 
\item[(a)] $\Sigma^n$ is $\frac{1}{L}$-outer almost minimizing in every $(U^1,U^2) \in \mathcal{CO}$.
\item[(b)] $\mathcal{H}^2(\Sigma^n \cap M) \geq m_0 -\frac{1}{L}$. 
\end{itemize}

\textit{Proof of Claim 1:} We define the sets of ``big slices'' for each $n>L$ by
\begin{equation*} 
K_n= \left\{ t \in [0,1] \; : \; \mathcal{H}^2(\Sigma^n_t \cap M) \geq m_0 -\frac{1}{L} \right\}. 
\end{equation*}
Note that $K_n$ is compact since $\{\Sigma^n_t\}$ is a continuous sweepout (by (4) in Definition 4.1). If Claim 1 is false, then for every $t \in K_n$, there exists a pair of open sets $(U^{1,t},U^{2,t}) \in \mathcal{CO}$ such that $\Sigma^n_t$ is not $\frac{1}{L}$-outer almost minimizing in any one of them. So for every $t \in K_n$, there exists isotopies $\{\varphi^{1,t}_s\}_{s \in [0,1]} \in \mathfrak{Is}_{\text{out}}(U^{1,t})$ and $\{\varphi^{2,t}_s\}_{s \in [0,1]} \in \mathfrak{Is}_{\text{out}}(U^{2,t})$ such that for $i=1,2$,
\begin{itemize}
	\item[(1)] $\mathcal{H}^2(\varphi^{i,t}_s(\Sigma^n_t)\cap M) \leq \mathcal{H}^2(\Sigma^n_t \cap M)+\frac{1}{8L}$ for every $s \in [0,1]$;
	\item[(2)] $\mathcal{H}^2(\varphi^{i,t}_1(\Sigma^n_t)\cap M) \leq \mathcal{H}^2(\Sigma^n_t \cap M)-\frac{1}{L}$.
\end{itemize} 

Next, we want to establish the following claim. 

\textit{Claim 2:} For each $t \in K_n$, there exists $\delta=\delta(t)>0$ such that if $|\tau-t|<\delta$, then for $i=1,2$, 
\begin{itemize}
	\item[(1')] $\mathcal{H}^2(\varphi^{i,t}_s(\Sigma^n_\tau) \cap M) \leq \mathcal{H}^2(\Sigma^n_\tau \cap M)+\frac{1}{4L}$ for every $s \in [0,1]$;
	\item[(2')] $\mathcal{H}^2(\varphi^{i,t}_1(\Sigma^n_\tau) \cap M) \leq \mathcal{H}^2(\Sigma^n_\tau \cap M)-\frac{1}{2L}$.
\end{itemize}

\textit{Proof of Claim 2:} To see why (1') is true, we argue by contradiction. Suppose no such $\delta$ exists, then there exists a sequence $\tau_j \to t$ and $s_j \in [0,1]$ such that for all $j$, 
\begin{equation*} 
\mathcal{H}^2(\varphi^{i,t}_{s_j}(\Sigma^n_{\tau_j}) \cap M) > \mathcal{H}^2(\Sigma^n_{\tau_j} \cap M)+\frac{1}{4L}. 
\end{equation*}
After passing to a subsequence, we can assume that $s_j \to s_0$ for some $s_0 \in [0,1]$. Observe that $\varphi^{i,t}_{s_j}(\Sigma^n_{\tau_j})$ converges weakly as varifolds to $\varphi^{i,t}_{s_0}(\Sigma^n_t)$ as $j \to \infty$. By (2.6.2(c)) in \cite{Allard72} and the fact that the sweepouts $\{\Sigma^n_t\}$ are continuous, we have 
\begin{align*} 
\mathcal{H}^2(\varphi^{i,t}_{s_0}(\Sigma^n_t) \cap M) &\geq \limsup_{j \to \infty} \mathcal{H}^2(\varphi^{i,t}_{s_j}(\Sigma^n_{\tau_j}) \cap M) \\
&\geq \lim_{j \to \infty} \mathcal{H}^2(\Sigma^n_{\tau_j} \cap M)+\frac{1}{4L} \\
&=\mathcal{H}^2(\Sigma^n_t \cap M)+\frac{1}{4L}.
\end{align*}
This contradicts (1) above. So we can choose $\delta>0$ such that (1') holds.

The proof of (2') is similar. Again, if no such $\delta>0$ exists, then there exists a sequence $\tau_j \to t$ such that for all $j$,
\begin{equation*} 
\mathcal{H}^2(\varphi^{i,t}_1(\Sigma^n_{\tau_j}) \cap M) > \mathcal{H}^2(\Sigma^n_{\tau_j} \cap M)-\frac{1}{2L}. 
\end{equation*}
Since $\varphi^{i,t}_1(\Sigma^n_{\tau_j})$ converges weakly to $\varphi^{i,t}_1(\Sigma^n_t)$ in $\tilde{M}$ as $j \to \infty$, thus, we have 
\begin{align*} 
\mathcal{H}^2(\varphi^{i,t}_1(\Sigma^n_t) \cap M) &\geq \limsup_{j \to \infty} \mathcal{H}^2(\varphi^{i,t}_1(\Sigma^n_{\tau_j}) \cap M) \\
&\geq \lim_{j \to \infty} \mathcal{H}^2(\Sigma^n_{\tau_j} \cap M) -\frac{1}{2L} \\
&=\mathcal{H}^2(\Sigma^n_t \cap M)-\frac{1}{2L}.
\end{align*}
This contradicts (2) above. Therefore, (2') holds for some $\delta>0$. Thus, Claim 2 is established.

The rest of the proof is exactly the same as in section 5.2 of \cite{Colding-DeLellis03}, replacing $\mathcal{H}^2$ by $\mathcal{H}^2(\cdot \cap M)$ and (2'), (3') in \cite{Colding-DeLellis03} by our (1'), (2') in Claim 2 above. Note that there is a typo in the last line of equation in that section. It should read
\begin{equation*}
\mathcal{H}^2(\Gamma^n_t) \leq \mathcal{H}^2(\Sigma^n_t)-\frac{1}{2L}+\frac{1}{4L} \leq \mathcal{F}(\{\Sigma^n_t\})- \frac{1}{4L}.
\end{equation*}
The proof of Claim 1 is completed and so is Lemma 6.3.
\end{proof}

For any $x \in \tilde{M}$, $0<s<t$, let An$(x,s,t)$ denote the open annulus centered at $x$ with inner radius $s$ and outer radius $t$. Let $r>0$, we set $\mathcal{A}_r(x)$ as the collection of open annuli An$(x,s,t)$ such that $0<s<t<r$. By the same argument as in Proposition 5.1 of Colding-De Lellis \cite{Colding-DeLellis03}, we obtain Proposition 4.10, which we restate below.

\begin{proposition}[Proposition 4.10]
Given a saturated collection $\Lambda$, there exists a positive function $r:\tilde{M} \to \mathbb{R}$ and a min-max sequence of surfaces  $\{\Sigma^n=\Sigma^n_{t_n} \cap M\}_{n \in \mathbb{N}}$ such that 
\begin{itemize}
	\item[(1)] $\{\Sigma^n\}_{n \in \mathbb{N}}$ is an outer almost minimizing sequence in any annulus An $\in \mathcal{A}_{r(x)}(x)$, where $x$ is any point in $\tilde{M}$;
	\item[(2)] In every such annulus An, $\Sigma^n$ is a smooth surface (possibly with boundary) when $n$ is sufficiently large;
	\item[(3)] The sequence $\Sigma^n$ converges to a freely stationary varifold $V$ in $M$ as $n \to \infty$.
\end{itemize}
\end{proposition}


\section{A minimization problem with partially free boundary}

In this section, we prove a result about minimizing area in $M$ among isotopic surfaces similar to the ones obtained by F. Almgren and L. Simon \cite{Almgren-Simon79}, W. Meeks, L. Simon and S.T. Yau \cite{Meeks-Simon-Yau82}, M. Gr\"{u}ter and J. Jost \cite{Gruter-Jost86} and J. Jost \cite{Jost86}. Since  we are restricting to the class of outward isotopies, we need to modify some of the arguments used in the papers above.

First, we define the concept of \emph{admissible open sets}.

\begin{definition}
An open set $U \subset \tilde{M}$ is said to be \emph{admissible} if it satisfies all the following properties:
\begin{itemize}
\item[(i)] $U$ is smooth, i.e. $U$ is an open set with smooth boundary $\partial U$;
\item[(ii)] $U$ is uniformly convex in the sense that all the principal curvatures with respect to the inward normal is positive along $\partial U$;
\item[(iii)] the closure $\overline{U}$ is diffeomorphic to the closed unit 3-ball in $\mathbb{R}^3$;
\item[(iv)] $\partial U$ intersects $\partial M$ transversally and $U \cap \partial M$ is topologically an open disk;
\item[(v)] the angle between $\partial U$ and $\partial M$ is always less than $\pi/2$ when measured in $U \cap M$, i.e. if $\nu_U$ and $\nu_M$ are the outward unit normal of $U$ and $M$ respectively, then $\nu_U \cdot \nu_M <0$ along $\partial U \cap \partial M$. 
\end{itemize}
\end{definition}

Given a surface $\Sigma$ in $\tilde{M}$, we want to minimize area (in $M$) among all the surfaces which are outward isotopic to $\Sigma$ and are identical to $\Sigma$ outside an admissible open set $U$.

\begin{definition}
Let $U \subset \tilde{M}$ be an admissible open set. Let $\Sigma \subset \tilde{M}$ be an embedded closed surface (not necessarily connected) intersecting $\partial U$ transversally. Consider the minimization problem $(\Sigma,\mathfrak{Is}_{\text{out}}(U))$: 
\begin{equation*} 
\alpha = \inf_{\{\varphi_s\} \in \mathfrak{Is}_{\text{out}}(U)} \mathcal{H}^2(\varphi_1(\Sigma) \cap M),
\end{equation*}
if a sequence $\{\varphi^k_s\}_{k \in \mathbb{N}} \in \mathfrak{Is}_{\text{out}}(U)$ satisfies
\begin{equation*} 
\lim_{k \to \infty} \mathcal{H}^2(\varphi^k_1(\Sigma) \cap M)=\alpha, 
\end{equation*}
we say that $\Sigma_k=\varphi^k_1(\Sigma)$ is a \emph{minimizing sequence} for the minimization problem $(\Sigma,\mathfrak{Is}_{\text{out}}(U))$. 
\end{definition}

Note that if two surfaces $\Sigma_1$ and $\Sigma_2$ agree in $M$, i.e. $\Sigma_1 \cap M=\Sigma_2 \cap M$, then the minimization problems $(\Sigma_1,\mathfrak{Is}_{\text{out}}(U))$ and $(\Sigma_2,\mathfrak{Is}_{\text{out}}(U))$ are equivalent since we only count area in $M$ and $\varphi_1(\Sigma_1) \cap M=\varphi_1(\Sigma_2) \cap M$ for any $\{\varphi_s\}_{s \in [0,1]} \in \mathfrak{Is}_{\text{out}}$.

By a small perturbation by outward isotopies, we can always obtain a minimizing sequence $\Sigma_k$ so that $\Sigma_k$ intersects $\partial M$ transversally for every $k$. The key result of this section is the following theorem.

\begin{theorem}
Let $U \subset \tilde{M}$ be an admissible open set, suppose $\{\Sigma_k\}$ is a minimizing sequence for the minimization problem $(\Sigma,\mathfrak{Is}_{\text{out}}(U))$ defined in Definition 7.2 so that $\Sigma_k$ intersects $\partial M$ transversally for each $k$, and $\Sigma_k \cap M$ converge weakly to a varifold $V \in \mathcal{V}(M)$.

Then, the following holds:
\begin{itemize}
\item[(a)] $V=\Gamma$, for some compact embedded minimal surface $\Gamma \subset \overline{U} \cap M$ with smooth boundary (except possibly at $\partial U \cap \partial M$) contained in $\partial (U \cap M)$;
\item[(b)] The fixed boundary of $\Gamma$ is the same as  $\Sigma \cap (\partial U \cap M)$; 
\item[(c)] $\Gamma$ meets $\partial M \cap U$ orthogonally along the free boundary of $\partial \Gamma$; 
\item[(d)] $\Gamma$ is stable with respect to $\mathfrak{Is}_{\text{tan}}(U)$.
\end{itemize}
\end{theorem}

\begin{remark}
It is clear that $V$ is freely stationary. The only thing we have to prove is regularity. Interior regularity follows from a localized version of \cite{Meeks-Simon-Yau82} given in Proposition 3.3 of \cite{DeLellis-Pellandini10}. The regularity at the fixed boundary $\partial U \cap $int$(M)$ is also discussed in \cite{DeLellis-Pellandini10}. Therefore, the only case left is the regularity at the free boundary $\partial M \cap U$. Hence, Theorem 7.3 says that the limit varifold $V$ is equal to a stable, smooth, properly embedded minimal surface $\Gamma$ (possibly disconnected).
\end{remark}

The proof of Theorem 7.3 goes as follows. We first apply a version of local $\gamma$-reduction (see \cite{Meeks-Simon-Yau82} and \cite{DeLellis-Pellandini10}) to reduce the minimization problem to the case of genus zero surfaces. Then we use a result from \cite{Jost86} to conclude that such minimizers are smooth.

\subsection{Local $\gamma$-reductions}

Following \cite{DeLellis-Pellandini10} and \cite{Meeks-Simon-Yau82}, replacing area by area \emph{in $M$}, we modify some of their  definitions and propositions for our purpose. First of all, we fix $\delta>0$ such that the following lemma holds (see Lemma 4.2 of \cite{Jost86}). 

\begin{lemma}
There exists $r_0>0$ and $\delta \in (0,1)$, depending only on $\tilde{M}$ and $M$, with the property that if $\Sigma$ is a surface in int($M$) with $\partial \Sigma \subset \partial M$ and 
\begin{equation*}
\mathcal{H}^2(\Sigma \cap B_{r_0}(x)) < \delta^2 r_0^2 \qquad \text{for each } x \in \tilde{M}
\end{equation*}
then there exists a unique compact set $K \subset M$ such that 
\begin{itemize}
\item[(a)] $\partial K \cap $ int$(M)=\Sigma$ (i.e. $K$ is bounded by $\Sigma$ modulo $\partial M$);
\item[(b)] $\mathcal{H}^3(K \cap B_{r_0}(x)) \leq \delta^2 r_0^3 \qquad \text{for each } x \in \tilde{M}$; and 
\item[(c)] $\mathcal{H}^3(K) \leq c_0 \mathcal{H}^2(\Sigma)^{\frac{3}{2}}$, where $c_0$ depends only on $\tilde{M}$ and $M$.
\end{itemize}
\end{lemma}

By rescaling the metric of $\tilde{M}$ if necessary, we can assume that $r_0=1$ in Lemma 7.5. From now on, we will assume that $\delta>0$ satisfies Lemma 7.5 with $r_0=1$. We will generalize the notion of $\gamma$-reductions to allow boundary reductions as well. Suppose $0<\gamma<\delta^2/9$.

\begin{definition}
Let $\Sigma_1$ and $\Sigma_2$ be closed (possibly disconnected) embedded surfaces in $\tilde{M}$. We say that \emph{$\Sigma_2$ is a $(\gamma,U)$-reduction of $\Sigma_1$} and write
\begin{equation*}
\Sigma_2 \stackrel{(\gamma,U)}{\ll} \Sigma_1
\end{equation*}
if the following conditions are satisfied:
\begin{itemize}
	\item[(1)] $\Sigma_2$ is obtained from $\Sigma_1$ through a surgery in $U$, that is,
		\begin{itemize}
			\item[(i)] $\overline{\Sigma_1 \setminus \Sigma_2} \cap M= A \subset U$ is diffeomorphic to either a closed annulus $\mathcal{A}=\{(x_1,x_2)\in \mathbb{R}^2| 1\leq x_1^2+x_2^2 \leq 2\}$ or a closed half-annulus $\mathcal{A}_+=\{ (x_1,x_2) \in \mathcal{A}| x_2 \geq 0\}$;
			\item[(ii)] $\overline{\Sigma_2 \setminus \Sigma_1} \cap M=D_1 \cup D_2 \subset U$ with each $D_i$ diffeomorphic to either the closed unit disk $\mathcal{D}=\{(x_1,x_2) \in \mathbb{R}^2 | x_1^2+x_2^2 \leq 1\}$ or the closed unit half-disk $\mathcal{D}_+=\{(x_1,x_2)\in \mathcal{D}| x_2 \geq 0\}$;
			\item[(iii)] There exists a compact set $K$ embedded in $U$, homeomorphic to the closed unit 3-ball with $\partial K=A\cup D_1 \cup D_2$ modulo $\partial M$ (i.e. there exists a compact set $Y \subset \partial M$ such that $\partial K=A \cup D_1 \cup D_2 \cup Y$), and $(K \setminus \partial K) \cap (\Sigma_1 \cup \Sigma_2)=\emptyset$.
		\end{itemize}
	\item[(2)] $\mathcal{H}^2(A)+\mathcal{H}^2(D_1)+\mathcal{H}^2(D_2) < 2\gamma$;
	\item[(3)] If $\Gamma$ is a connected component of $\Sigma_1 \cap \overline{U} \cap M$ containing $A$, and $\Gamma \setminus A$ is disconnected, then for each component of $\Gamma \setminus A$ we have one of the following possibilities:
		\begin{itemize}
			\item[(a)] either it is a genus zero surface contained in $U \cap M$ with area  $\geq \delta^2/2$;
			\item[(b)] or it is not a genus zero surface.
		\end{itemize}
\end{itemize}
We say that $\Sigma$ is \emph{$(\gamma,U)$-irreducible} if there does not exist $\tilde{\Sigma}$ such that $\tilde{\Sigma} \stackrel{(\gamma,U)}{\ll} \Sigma$.
\end{definition}

A immediate consequence of the above definition is the following.

\begin{remark}
$\Sigma$ is $(\gamma,U)$-irreducible if and only if whenever $\Delta \subset U \cap M$ is a closed disc or half-disk with $\partial \Delta \setminus \partial M=\Delta \cap \Sigma$ and $\mathcal{H}^2(\Delta) < \gamma$, then there is a closed genus 0 surface $D \subset \Sigma \cap U \cap M$ with $\partial \Delta \setminus \partial M=\partial D \setminus \partial M$ and $\mathcal{H}^2(D)<\delta^2/2$.
\end{remark}

Similar to \cite{Meeks-Simon-Yau82}, we define strong $(\gamma,U)$-reduction as follows.

\begin{definition}
Let $\Sigma_1$ and $\Sigma_2$ be closed (possibly disconnected) embedded surfaces in $\tilde{M}$. We say that \emph{$\Sigma_2$ is a strong $(\gamma,U)$-reduction of $\Sigma_1$} and write
\begin{equation*}
\Sigma_2 \stackrel{(\gamma,U)}{<} \Sigma_1
\end{equation*}
if there exists an isotopy $\{\psi_s\}_{s \in [0,1]} \in \mathfrak{Is}_{\text{out}}(U)$ such that 
\begin{itemize}
	\item[(1)] $\Sigma_2 \stackrel{(\gamma,U)}{\ll} \psi_1(\Sigma_1)$;
	\item[(2)] $\Sigma_2 \cap (\tilde{M} \setminus U)=\Sigma_1 \cap (\tilde{M} \setminus U)$;
	\item[(3)] $\mathcal{H}^2((\psi_1(\Sigma_1) \Delta \Sigma_1)\cap M) < \gamma$.
\end{itemize}
We say that $\Sigma$ is \emph{strongly $(\gamma,U)$-irreducible} if there is no $\tilde{\Sigma}$ such that $\tilde{\Sigma} \stackrel{(\gamma,U)}{<} \Sigma$.
\end{definition}

Following the same arguments in Remark 3.1 of \cite{Meeks-Simon-Yau82}, we have the following proposition.

\begin{proposition}
Given any closed embedded surface $\Sigma$ (not necessarily connected), there exists a sequence $\Sigma=\Sigma_1,\Sigma_2,\ldots,\Sigma_k$ of closed embedded surfaces (not necessarily connected) such that 
\begin{equation*}
\Sigma_k \stackrel{(\gamma,U)}{<} \Sigma_{k-1} \stackrel{(\gamma,U)}{<} \cdots \stackrel{(\gamma,U)}{<}\Sigma_1=\Sigma
\end{equation*}
where $\Sigma_k$ is strongly $(\gamma,U)$-irreducible. Furthermore, there exists a constant $c>0$ which depends only on genus$(\Sigma \cap M)$ and $\mathcal{H}^2(\Sigma \cap M)/\delta^2$ so that $k \leq c$, and
\begin{equation*}
\mathcal{H}^2((\Sigma \Delta \Sigma_k) \cap M) \leq 3c\gamma.
\end{equation*}
\end{proposition}

\begin{proof}
The proof is the same as the proof of Remark 3.1 in \cite{Meeks-Simon-Yau82}.
\end{proof}

The following theorem gives our main result for strongly $(\gamma,U)$-irreducible surfaces $\Sigma$. For any closed surface $\Sigma$, we denote
\begin{equation*}
E(\Sigma)=\mathcal{H}^2(\Sigma \cap M) - \inf_{\Sigma' \in J_U(\Sigma)} \mathcal{H}^2(\Sigma' \cap M),
\end{equation*}
where $J_U(\Sigma)=\{\varphi_1(\Sigma): \{\varphi_s\}_{s \in [0,1]} \in \mathfrak{Is}_{\text{out}}(U)\}$ denotes the set of all surfaces which are outward isotopic to $\Sigma$ in $U$. Let $\Sigma_0$ denote the union of all components $\Lambda \subset \Sigma \cap \overline{U} \cap M$ such that there exists some $K_\Lambda \subset U$ diffeomorphic to the unit 3-ball such that $\Lambda \subset K_\Lambda$ and $\partial K_\Lambda \cap \Sigma \cap M=\emptyset$.

\begin{theorem}
Let $U \subset \tilde{M}$ be an admissible open set, and $A \subset U$ be a compact subset diffeomorphic to the unit 3-ball. Assume $\partial M$ intersects both $\partial U$ and $\partial A$ transversally. 

Suppose $\Sigma \subset \tilde{M}$ is a smooth closed embedded surface (possibly disconnected) such that
\begin{itemize}
\item[(i)] $\Sigma$ intersects both $\partial M$ and $\partial A$ transversally;
\item[(ii)] $E(\Sigma) \leq \gamma/4$ and is strongly $(\gamma,U)$-irreducible;
\item[(iii)] For each component $\Gamma$ of $\Sigma \cap \partial A \cap M$, let $F_\Gamma$ be a the component in $(\partial A \cap M) \setminus \Gamma$ such that $\partial F_\Gamma \setminus \partial M=\Gamma$ and $\mathcal{H}^2(F_\Gamma)= \min \{ \mathcal{H}^2(F_\Gamma),\mathcal{H}^2((\partial A \cap M) \setminus F_\Gamma)\}$. Furthermore, suppose that $\sum_{j=1}^q \mathcal{H}^2(F_j) \leq \frac{\gamma}{8}$, where $F_j=F_{\Gamma_j}$ and $\Gamma_1,\ldots,\Gamma_q$ denote the components of $\Sigma \cap \partial A \cap M$. Note that each $\Gamma$ is either a closed Jordan curve in $M$ or a Jordan arc with endpoints on $\partial M$, and each $F_\Gamma$ is either a disk, a half-disk or an annulus in $M$. 
\end{itemize}
Then, $\mathcal{H}^2(\Sigma_0) \leq E(\Sigma)$ and there exists pairwise disjoint, connected, closed genus zero surfaces $D_1,\ldots,D_p$ with
$D_i \subset (\Sigma \setminus \Sigma_0) \cap U \cap M$, $\partial D_i \setminus \partial M \subset \partial A$ and $\left(\cup_{i=1}^p D_i \right) \cap A \cap M = (\Sigma \setminus \Sigma_0) \cap A \cap M$. Moreover,
\begin{equation*}
\sum_{i=1}^p \mathcal{H}^2(D_i) \leq \sum_{j=1}^q \mathcal{H}^2(F_j)+ E(\Sigma),
\end{equation*}
Furthermore, for any given $\alpha>0$, we have
\begin{equation*}
\mathcal{H}^2((\cup_{i=1}^p (\varphi_1(D_i) \setminus \partial D_i)) \cap M \setminus (A \setminus \partial A)) < \alpha
\end{equation*}
for some isotopy $\{\varphi_s\}_{s \in [0,1]} \in \mathfrak{Is}_{\text{out}}(U)$ (depending on $\alpha$) which is identity on some open neighborhood of $(\Sigma \setminus \Sigma_0)\setminus \cup_{i=1}^p (D_i \setminus \partial D_i)$.
\end{theorem}

Although all sufficiently small disks with the same boundary are isotopic to each other, which is a crucial point in the proof of Theorem 2 in \cite{Meeks-Simon-Yau82}, not all of them are outward isotopic to each other. Therefore, we need to establish the lemma below, which says that all small disks are almost outward isotopic, modulo arbitrarily small area.

\begin{lemma}
Let $U$ and $A$ be as in Theorem 7.10. Let $\Gamma$ be a Jordan curve in $\partial A \cap $int$(M)$ which is either closed or having endpoints on $\partial M$. Let $F \subset \partial A$ be a connected component of $(\partial A \cap M) \setminus \Gamma$, which is diffeomorphic to a disk, a half-disk or an annulus. Let $D \subset U \cap M$ be a genus zero surface transversal to $\partial M$ with $\partial D \setminus \partial M=\partial F \cap \setminus \partial M=\Gamma$, and $D \cap F=\emptyset$. In addition, we assume that $F \cup D$ bounds a unique compact set $K \subset M$ modulo $\partial M$, i.e. $\partial K \cap $int$(M)=F \cup D$.

Then, for any $\alpha >0$, there exists an isotopy $\{\varphi_s\}_{s \in [0,1]} \in \mathfrak{Is}_{\text{out}}(U)$ supported on a small neighborhood of $K$ such that $\varphi_s(x)=x$ for all $x \in \Gamma$ and all $s \in [0,1]$, moreover, 
\begin{equation*}
\mathcal{H}^2((\varphi_1(D) \cap M)  \Delta F) < \alpha.
\end{equation*}
In other words, we can outward isotope $D$ to approximate $F$ as close as we want.
\end{lemma}

\begin{proof}

We divide the situation into two cases according to whether the boundary curve $\Gamma$ is closed or not.

\textit{Case 1:} $\Gamma$ is a simple closed curve. 

In this case, $F$ is either a closed disk or a closed annulus in $\partial A \cap M$. In the latter case, we will show that we can even find an isotopy $\{\varphi_s\}$ supported on a neighborhood of $K$, leaving $\Gamma$ fixed, and $\varphi_1(D) \cap M=F$.

After a change of coordinate, we can assume that
\begin{itemize}
\item $U$ is the open ball of radius $2$ in $\mathbb{R}^3$ centered at origin; 
\item $A \subset U$ is the closed unit 3-ball centered at origin;
\item $M \cap U=U \cap \{x_3 \geq 0\}$ is the upper half-ball; and
\item $\Gamma=\{(x_1,x_2,x_3) \in \mathbb{R}^3\; : \; x_3= \frac{1}{2}, \; x_1^2 + x_2^2=\frac{3}{4}\}$.
\end{itemize}

First, we look at the case that $F$ is a closed disk, i.e. $F=\{(x_1,x_2,x_3) \in \mathbb{R}^3 \; : \; x_3 \geq \frac{1}{2} \}$. Let $D$ be the genus zero surface as given in the hypothesis. Note that $D$ meets $\partial M$ at a finite number of simple closed curves $\Gamma_i$, $i=1,\ldots,N$, each of which bounds a closed disk $D_i$ in $\partial M \cap U$. Since $D$ is a genus zero surface with boundary, it is clear that $D \cup F \cup (D_1 \cup \cdots \cup D_N)$ is homeomorphic to a 2-sphere, and thus, the compact set $K$ bounded by $F \cup D$ modulo $\partial M$ is homeomorphic to the unit 3-ball in $\mathbb{R}^3$. Observe that if $D \cap \partial M =\emptyset$, then it is trivial that we can isotope $D$ to $F$, holding $\partial M$ fixed. If $D \cap \partial M \neq \emptyset$, we then use a tangential isotopy to deform it so that it approximates a disk disjoint from $\partial M$ with boundary $\Gamma$, which in turn is isotopic to $F$. To see this, perturb each closed disk $D_i$ into the interior of $M$ such that the boundary of the disk stays on $D$, call $\hat{D}_i$ the perturbed disk, then there is a closed annulus $A_i \subset D$ such that $\partial A_i = \partial D_i \cup \partial \hat{D}_i$ and $A_i \cup D_i \cup \hat{D}_i$ bounds a ball in $K$. Using a tangential isotopy, we can deform $D$ such that it agrees with $(D \setminus A_i) \cup \hat{D}_i$ with an error in area as small as we want (one simply shrinks the size of the neck in $K$). Repeat the whole procedure for each $D_i$, one can deform $D$ such that is arbitrarily close to a disk disjoint from $\partial M$, and we are done.

In the case that $F$ is a closed annulus, again let $\Gamma_i$ be the set of simple closed curves where $D$ meets $\partial M$. Note that all except possibly one $\Gamma_i$ bounds a disk in $K$. For those which bounds a disk, we can repeat the ``neck-shrinking'' argument as in the previous case to eliminate them. Therefore, we can assume that $\Gamma_i$ together with $\Gamma_0=\partial F \cap \partial M$ bounds a connected genus 0 surface $\Lambda$ in $\partial M$. After a further change of coordinate, we can assume that $K=\Lambda \times [0,\frac{1}{2}]$, where we think of $\Lambda$ as a subset of $\partial M \subset \mathbb{R}^2$. Next consider the outward isotopy given by the vertical translations $\varphi_s(x_1,x_2,x_3)=(x_1,x_2,x_3-s)$ (with some cutoff near $\partial U$ so that it is supported in $U$), take a smooth function $\chi$ on $\mathbb{R}^2$ such that $\chi=0$ along $\Gamma_0$, $\chi=1$ outside a small neighborhood $V$ of $\Gamma_0$ disjoint from all $\Gamma_i$, and $0<\chi<1$ elsewhere. By choosing the neighborhood $V$ of $\Gamma_0$ smaller and smaller, we see that the outward isotopy $\{\chi \varphi_s\}$ given by the vertical translations with cutoff would deform $D$ to approximate $F$ as close as we want. This implies our desired conclusion.

\textit{Case 2:} $\Gamma$ is an arc with endpoints on $\partial M$.

Assume the standard setting as before after a change of coordinate, suppose for our convenience that $\Gamma=\{(x_1,x_2,x_3) \in \partial A \; : \; x_1=0\}$, and $F=\{(x_1,x_2,x_3) \in \partial A \; : \; x_1 \geq 0\}$. Note that $D$ intersects $\partial M$ at a Jordan arc $\Gamma_1$ with the same endpoints as $\Gamma$ and a (possibly empty) finite collection of disjoint simple closed curves $\Gamma_i$, $i=2,\ldots,N$. Let $\Gamma_0=F \cap \partial M$. By assumption, there is a compact set $K \subset M \cap U$ such that $\partial K=D \cup F \cup \Lambda$ where $\Lambda$ is a genus zero surface in $\partial M$ with $\partial \Lambda=\cup_{i=0}^N \Gamma_i$. For all those $\Gamma_i$, $i=2,\ldots,N$, which bounds a disk in $\Lambda$, we can shrink down the neck as in Case 1. So we can assume without loss of generality that $\Lambda$ is connected. There are two further sub-cases: either $\Gamma_0 \cup \Gamma_1$ is the outermost boundary of $\Lambda$ or it is not. In the first case, similar to the second part of Case 1 above, we can assume (up to a change of coordinate) that $F=\{(x_1,x_2,x_3) \in \mathbb{R}^3 : x_1=0, x_3 \geq 0, x_2^2 +x_3^2 \leq 1\}$, $\Lambda=\{(x_1,x_2,x_3) \in \mathbb{R}^3: x_3=0,x_1 \geq 0, x_1^2+x_2^2 \leq 1\} \setminus \cup_{i=2}^N D_i$ where $D_i$ is the disk in $\partial M$ bounded by the simple closed curve $\Gamma_i$, and $D$ is the union of a graph over $\Lambda$ and $n-1$ cylinders contained in $\Gamma_i \times [0,1]$. Using a cutoff function $\chi$ as before which is zero on $\Gamma_0=F \cap \partial M$ and the vertical translations, we can deform $D$ to approximate $F$ as close as we want. Now we are left with the case that $\Gamma_0 \cup \Gamma_1$ is not the outermost boundary of $\Lambda$. In this case, we can take $\Lambda=\{(x_1,x_2,x_3) \in \mathbb{R}^3 : x_3=0, x_1^2+x_2^2 \leq \frac{9}{16}\} \setminus \cup_{i=2}^N D_i$ where $D_i$ is the disk in $\partial M$ bounded by the simple closed curve $\Gamma_i$, and $D$ is the union of a graph over $\Lambda$ and $n-1$ cylinders contained in $\Gamma_i \times [0,1]$. Then a similar translation with cutoff will deform $D$ to approximate $F$ and we are done.

\end{proof}

Now we are ready to give a proof of Theorem 7.10.

\begin{proof}[Proof of Theorem 7.10]
The proof is almost the same as that in \cite{Meeks-Simon-Yau82}, except that we have to use Lemma 7.11 because not all small disks are outward isotopic to each other. 

As in Meeks-Simon-Yau \cite{Meeks-Simon-Yau82}, we can assume that $\Sigma_0=\emptyset$. We proceed by induction on $q$. Denote
\begin{itemize}
	\item[$(H)_q$] $\Sigma_0=\emptyset$, $\sum_{j=1}^q \mathcal{H}^2(F_j) \leq \gamma/8$, $E(\Sigma) \leq \gamma/2 -2\sum_{j=1}^q \mathcal{H}^2(F_j)$, and $\Sigma$ is strongly $(\hat{\gamma},U)$-irreducible, where $\hat{\gamma}=\gamma/4 +4 \sum_{j=1}^q \mathcal{H}^2(F_j) + E(\Sigma)$.
	\item[$(C)_q$] the conclusion of Theorem 7.10 is true.
\end{itemize}
We will show that the statement ``$(H)_q \Rightarrow (C)_q$'' is true for all $q$. Assume it is true for $q-1$. We want to show by induction that it is true for $q$ also.

Relabeling if necessary, we can assume that $F_q$ is innermost, i.e. $F_q \cap \Gamma_j=\emptyset$ for all $j \neq q$. Since $\Sigma$ is strongly $(\hat{\gamma},U)$-irreducible and $\mathcal{H}^2(F_q) < \hat{\gamma}$. By Remark 7.7, there exists a connected genus zero surface $D \subset \Sigma \cap U \cap M$ such that $\partial D \setminus \partial M=\Gamma_q$ and $\mathcal{H}^2(D) < \delta^2/2$. Since $F_q$ is innermost, $D \cap F_q=\emptyset$. As the area of $F_q$ and $D$ are small, we can apply Lemma 7.5. Hence, there is a unique compact set $K \subset M$ which is bounded by $F_q \cup D$ modulo $\partial M$. Moreover, since $F \cup D$ is a genus zero surface, it is easy to see that there exists a small neighborhood of $K$ which is diffeomorphic to the unit 3-ball whose boundary is disjoint from $\Sigma \cap M$. Because we assume that $\Sigma_0=\emptyset$, we know that the whole neighborhood is disjoint from $(\Sigma \cap M) \setminus D$.
  
Replace $D$ by $F_q$ and write $\Sigma_*=(\Sigma \setminus D) \cup F_q$ (which is only a Lipschitz surface), and $F_{q,\epsilon}=\{x \in \tilde{M}:d(x,F_q)<\epsilon\}$, for each $\epsilon>0$, we can select a continuous tangential isotopy $\{\varphi_s\}_{s \in [0,1]} \in \mathfrak{Is}_{\text{tan}}(U)$ such that $\varphi_s(F_{q,\epsilon}) \subset F_{q,\epsilon}$, $\varphi_s(x)=x$ for $x \notin F_{q,\epsilon}$, furthermore, 
\begin{equation}
\mathcal{H}^2((\Sigma_* \cap F_{q,\epsilon}) \cap M) \leq \mathcal{H}^2(\varphi_1(\Sigma_* \cap F_{q,\epsilon}) \cap M) \leq \mathcal{H}^2((\Sigma_* \cap F_{q,\epsilon}) \cap M) + \epsilon
\end{equation}
and
\begin{equation}
\varphi_1(\Sigma_* \cap F_{q,\epsilon}) \cap \partial A=\emptyset.
\end{equation}
In other words, we deform $\Sigma_*$ to detach $F_q$ from $\partial A$. Let $\hat{\Sigma}_*=\varphi_1(\Sigma_*)$ (smooth by suitably choosing $\varphi_1$), for $\epsilon$ small enough, we have
\begin{itemize}
	\item[(i)] $\hat{\Sigma}_* \cap \partial A=\cup_{j=1}^{q-1} \Gamma_j$,
	\item[(ii)] $\mathcal{H}^2((\hat{\Sigma}_* \Delta \Sigma)\cap M) < \mathcal{H}^2(D) + \mathcal{H}^2(F_q) + \epsilon$,
	\item[(iii)] $\mathcal{H}^2(\hat{\Sigma}_* \cap M) < \mathcal{H}^2(\Sigma \cap M) + \mathcal{H}^2(F_q) - \mathcal{H}^2(D) + \epsilon$.
\end{itemize}
Notice (iii) implies
\begin{itemize}
	\item[(iii)'] $E(\hat{\Sigma}_*) < E(\Sigma) + \mathcal{H}^2(F_q) - \mathcal{H}^2(D) + \epsilon$.
\end{itemize}
because Lemma 7.11 implies that
\begin{equation}
\inf_{\Sigma' \in J_U(\Sigma)} \mathcal{H}^2(\Sigma' \cap M) \leq \inf_{\Sigma' \in J_U(\hat{\Sigma}_*)} \mathcal{H}^2(\Sigma' \cap M).
\end{equation}

Taking $\epsilon < \mathcal{H}^2(F_q)$, following the same arguments in \cite{Meeks-Simon-Yau82}, we see that $\hat{\Sigma}_*$ satisfies $(H)_{q-1}$, hence $(C)_{q-1}$ holds by induction hypothesis. There must be pairwise disjoint connected genus zero surfaces $\tilde{\Delta}_1,\ldots,\tilde{\Delta}_p$ contained in $\hat{\Sigma}_* \cap U \cap M$ with $\partial \tilde{\Delta}_j \setminus \partial M \subset \partial A$, $\left( \cup_{i=1}^p \tilde{\Delta}_i \right) \cap (A \setminus \partial A) = \hat{\Sigma}_* \cap (A \setminus \partial A) \cap M$, and 
\begin{equation}
\sum_{i=1}^p \mathcal{H}^2(\tilde{\Delta}_i) \leq \sum_{j=1}^{q-1} \mathcal{H}^2(F_j) + E(\hat{\Sigma}_*).
\end{equation}
Furthermore, for any $\alpha>0$,
\begin{equation}
\mathcal{H}^2(\cup_{i=1}^p(\tilde{\Psi}_1(\tilde{\Delta}_i) \setminus \partial \tilde{\Delta}_i) \cap M \setminus (A \setminus \partial A)) < \frac{\alpha}{2} 
\end{equation} 
for some isotopy $\{\tilde{\Psi}_s\}_{s \in [0,1]} \in \mathfrak{Is}_{\text{out}}(U)$ which fixes a neighborhood of $(\hat{\Sigma}_* \cap M) \setminus \cup_{i=1}^p (\tilde{\Delta}_i \setminus \partial \tilde{\Delta}_i)$. Reversing the isotopy $\varphi$ used in (7.4) and (7.5), there are pairwise disjoint connected genus zero surfaces $\Delta_1,\ldots,\Delta_p \subset \Sigma_* \cap M=((\Sigma \cap M) \setminus D) \cup F_q$ with $\left( \cup_{i=1}^p \Delta_i \right) \cap (A \setminus \partial A)=\Sigma_* \cap (A \setminus \partial A) \cap M$, $\partial \Delta_i \setminus \partial M=\partial \tilde{\Delta}_i \setminus \partial M$, and
\begin{equation}
\sum_{i=1}^{p} \mathcal{H}^2(\Delta_i) \leq \sum_{j=1}^{q-1} \mathcal{H}^2(F_j) + E(\Sigma) + \mathcal{H}^2(F_q) - \mathcal{H}^2(D),
\end{equation}
Furthermore,
\begin{equation}
\mathcal{H}^2(\cup_{i=1}^p(\Psi_1(\Delta_i) \setminus \partial \Delta_i) \cap M \setminus (A \setminus \partial A)) < \frac{3\alpha}{4}
\end{equation} 
for some isotopy $\{\Psi_s\}_{s\in [0,1]} \in \mathfrak{Is}_{\text{out}}(U)$ which fixes a neighborhood of $(\Sigma_* \cap M) \setminus (\cup_{i=1}^p(\Delta_i \setminus \partial \Delta_i) \cup F_{q,\epsilon})$. 

Recall that $K$ is the compact set in $U \cap M$ bounded by $D\cup F_q$, by Lemma 7.11, there exists a continuous isotopy $\{ \beta_s\} \in \mathfrak{Is}_{\text{out}}(U)$ supported on a neighborhood of $K$ fixing $\Gamma_q=\partial D \setminus \partial M$ and 
\begin{equation}
\mathcal{H}^2( (\beta_1(D) \Delta F_q) \cap M) < \frac{\alpha}{8}.
\end{equation} 
Moreover, we know that $(\Sigma \setminus D) \cap K=\emptyset$ because $\Sigma_0=\emptyset$. Consider the following two cases: (i) $F_q \subset \cup_{i=1}^p \Delta_i$; and (ii) $F_q \not \subset \cup_{i=1}^p \Delta_i$.

In case (i), if $F_q \subset \cup_{i=1}^p \Delta_i$, by taking $D_{j_0}=(\Delta_{j_0} \setminus F_q) \cup D$ for the unique $j_0$ such that $F_q \subset \Delta_{j_0}$, and we select $D_j=\Delta_j$ for all $j \neq j_0$. Also, we define a continuous outward isotopy $\hat{\varphi}=\{\hat{\varphi}_s\}$ by $\hat{\varphi}=\Psi * \beta$; by smoothing $\hat{\varphi}$ we obtain an outward isotopy $\varphi$ satisfying the required conditions. Here $\Psi * \beta$ is defined by $\Psi * \beta_s(x)=\beta_{2s}(x)$ if $0 \leq s \leq \frac{1}{2}$, and $\Psi * \beta_s(x)=\Psi_{2s-1}(\beta_1(x))$ if $\frac{1}{2} < s \leq 1$.

In case (ii), if $F_q \not \subset \cup_{i=1}^p \Delta_i$, we define the set of pairwise disjoint connected genus 0 surfaces $D_1,\ldots,D_{p+1}$ by setting $D_j=\Delta_j$, $j=1,\ldots,p$, and $D_{p+1}=D$. In this case, we define a continuous isotopy $\hat{\varphi}$ by setting $\hat{\varphi}=\hat{\beta}*(\Psi * \beta)$, where $\hat{\beta}=\{\hat{\beta}_s\}$ is a smooth outward isotopy such that $\hat{\beta}_s(x)=x$ for all $x \in (\Sigma \cap M) \setminus D$ and $s \in [0,1]$, and such that $\hat{\beta}_1(F_q) \cap M$ is a genus zero surface $\hat{D} \subset A \cap M$ with $\partial \hat{D} \setminus \partial M =\partial D \setminus \partial M$, $\hat{D} \cap \partial A = \partial D \cap \partial A$, and $\hat{D} \cap \Psi_s(\Sigma_*) =\Gamma_q$ for all $s \in [0,1]$. Such a $\hat{\beta}$ exists once we show the claim that in case (ii), there is a neighborhood $W$ of $\Gamma_q=\partial D \setminus \partial M$ such that $W \cap D \subset A$. Otherwise, we would have $W$ with $\Gamma_q \subset W$ and $W \cap (\Sigma \setminus D) \cap M \subset A \setminus \partial A$, and this would imply that $F_q \subset \cup_{i=1}^p \Delta_i$ since  $\left( \cup_{i=1}^p \Delta_i \right) \cap (A \setminus \partial A)=\Sigma_* \cap (A \setminus \partial A) \cap M$ (see the statement above (7.6)), thus contradicting we are in case (ii). By smoothing $\hat{\varphi}$ we then again obtain the required outward isotopy $\varphi$.

In each of the above cases, we have, by (7.6), that 
\begin{equation*}
\sum_{i=1}^{p} \mathcal{H}^2(\Delta_i) \leq \sum_{j=1}^{q-1} \mathcal{H}^2(F_j) + E(\Sigma) + \mathcal{H}^2(F_q) - \mathcal{H}^2(D),
\end{equation*}
and hence
\begin{eqnarray*}
\sum_{i=1}^p \mathcal{H}^2(D_i) & \leq &  \sum_{j=1}^{q-1} \mathcal{H}^2(F_j) + E(\Sigma) + \mathcal{H}^2(F_q) - \mathcal{H}^2(D) + \mathcal{H}^2(D)\\
&=&  \sum_{j=1}^{q} \mathcal{H}^2(F_j) + E(\Sigma)
\end{eqnarray*}
This proves that statement $(C)_q$ and the proof is finished by induction.

\end{proof}

We will need a replacement lemma about finite collection of genus zero surfaces with disjoint boundaries (c.f. Lemma 2 in \cite{Meeks-Simon-Yau82}).

\begin{lemma}
Let $A \subset \tilde{M}$ be a closed subset which is diffeomorphic to the unit 3-ball such that $A \cap \partial M$ is diffeomorphic to the closed unit disk. Suppose $D_1,\ldots,D_R$ are connected genus zero surfaces in $A \cap M$ with $D_i \setminus \partial D_i \subset A \setminus \partial A$ and $\partial D_i \subset \partial A \cap M$. Also, assume that $(\partial D_i \setminus \partial M) \cap (\partial D_j \setminus \partial M)=\emptyset$ and that either $D_i \cap D_j=\emptyset$ or $D_i$ intersects $D_j$ transversally for all $i \neq j$.

Then, there exists pairwise disjoint connected genus zero surfaces $\tilde{D}_1, \ldots,\tilde{D}_R$ in $A \cap M$ with $\tilde{D}_ i \setminus \partial \tilde{D}_i \subset A \setminus \partial A$, $\partial \tilde{D}_i \cap \partial A=\partial D_i \cap \partial A$ and $\mathcal{H}^2(\tilde{D}_i) \leq \mathcal{H}^2(D_i)$ for $i=1,\ldots,R$. 

\end{lemma}

\begin{proof}
Assume that $R \geq 2$ and that $D_1,\ldots,D_{R-1}$ are already pairwise disjoint. If we can prove the required result in this case, then the general case follows by induction on $R$.

Let $\Gamma_1,\ldots,\Gamma_q$ be pairwise disjoint Jordan curves (either closed or have boundaries on $\partial M$), not necessarily connected, such that
\begin{equation}
D_R \cap \left( \cup_{i=1}^{R-1} D_i \right) = \cup_{j=1}^q \Gamma_j,
\end{equation}
also, for each $j=1,\ldots,R$, $\Gamma_j$ divides each $D_i$ which contains $\Gamma_j$ into two genus zero surfaces (maybe disconnected) $D_i \setminus \Gamma_j=D_i' \cup D_i''$ with $\partial D_i' \setminus \partial M = \Gamma_j = \partial D_i'' \setminus \partial M$ (since $\partial D_i \cap M$ are pairwise disjoint by assumption, either $D_i'$ or $D_i''$ is disjoint from $\partial D_i \setminus \partial M$) and,  as an inductive hypothesis, assume the lemma is true whenever (7.9) holds with $q$ replaced by $q-1$ on the right hand side (with $D_1,\ldots,D_{R-1}$ still being assumed pairwise disjoint).

For each $j=1,\ldots, q$, let $E_j$ be the part of $D_R \setminus \Gamma_j$ which is disjoint from $\partial D_R \setminus \partial M$. Hence, $\partial E_j \setminus \partial M = \Gamma_j$. Let $F_j$ be the corresponding part in $\cup_{i=1}^{R-1} D_i \setminus \Gamma_j$ which is disjoint from $\cup_{i=1}^{q-1} \partial D_i \setminus \partial M$. Hence, $\partial F_j \setminus \partial M = \Gamma_j$. Let $K \subset \cup_{i=1}^R D_i$ be a genus zero surface with $\partial K \setminus \partial M=\Gamma_{j_0}$ for some $j_0$ such that 
\begin{equation*}
\mathcal{H}^2(K) \leq \min_{j=1,\ldots,q} \{ \mathcal{H}^2(E_j), \mathcal{H}^2(F_j)\}.
\end{equation*}
Let $J \neq K$ be the other genus zero surface in $\cup_{i=1}^R D_i$ such that $\partial J \setminus \partial M=\partial K \setminus \partial M=\Gamma_{j_0}$. Evidently we must have
\begin{equation}
(K \setminus \partial K) \cap \left( \cup_{i \neq i_0} D_i \right) = \emptyset,
\end{equation}
where $i_0$ is such that $K \subset D_{i_0}$. Let $i_1 \neq i_0$ be such that $J \subset D_{i_1}$ (note that then one of $i_0, i_1$ is equal to $R$), and define $\hat{D}_j=D_j$ if $j \neq j_1$ and $\hat{D}_{i_1}=(D_{i_1} \setminus J) \cup K$. By (7.10) we have that each $\hat{D}_j$ is an embedded genus zero surface, and clearly $\partial \hat{D}_j \setminus \partial M = \partial D_j \setminus \partial M$, $\mathcal{H}^2(\hat{D}_j) \leq \mathcal{H}^2(D_j)$, $\hat{D}_1,\ldots,\hat{D}_R$ are pairwise disjoint and 
\begin{equation}
\hat{D}_R \cap \left( \cup_{i=1}^{R-1} \hat{D}_{i} \right) = K \cup \left( \cup_{j \neq j_0} \Gamma_j \right).
\end{equation}
By smoothing $\hat{D}_{i_1}$ near $\Gamma_{j_0}$ and making a slight perturbation near $K$, we then obtain genus 0 surfaces $\hat{D}_1^*,\ldots,\hat{D}_R^*$ with $\partial \hat{D}_j^* \setminus \partial M = \partial D_j \setminus \partial M$, $\mathcal{H}^2(\hat{D}_j^*) \leq \mathcal{H}^2(D_j)$, $\hat{D}_1^*,\ldots,\hat{D}_{R-1}^*$ pairwise disjoint, and (using (7.11)),
\begin{equation*}
\hat{D}_R^* \cap \left( \cup_{j=1}^{R-1} \hat{D}_j^* \right) = \cup_{j \neq j_0} \Gamma_j.
\end{equation*}
Hence, we can apply the inductive hypothesis to the collection $\{\hat{D}_j^*\}$, thus obtaining the required collection $\tilde{D}_1,\ldots,\tilde{D}_R$.

\end{proof}

\subsection{Minimizing sequence of genus zero surfaces}

In this section, we recall a result by J. Jost \cite{Jost86} on the regularity for minimizers of the minimization problem for genus zero surfaces with partially free boundary.

Let $A \subset \tilde{M}$ be an admissible open set. Let $\Gamma \subset \partial A \cap M$ be an embedded smooth curve in $M$ which either meets $\partial M$ at the two endpoints transversally or is disjoint from $\partial M$. Let $\mathcal{M}=\mathcal{M}(0,\Gamma)$ be the set of all genus zero surfaces $D$ contained in $M$ with $\Gamma$ as boundary modulo $\partial M$, i.e. $\partial D \setminus \partial M=\Gamma$, and which meets $\partial M$ transversally. We say that $D_k$ is a \emph{minimizing sequence for $\mathcal{M}$} if 
\begin{equation*}
\mathcal{H}^2(D_k) \leq \inf_{D \in \mathcal{M}} \mathcal{H}^2(D) + \epsilon_k
\end{equation*}
for some positive real numbers $\epsilon_k \to 0$ as $k \to \infty$.

\begin{theorem}[\cite{Jost86}]
Using the notation above, let $D_k \in \mathcal{M}$ be a minimizing sequence for $\mathcal{M}$, and suppose $D_k$ converges to $V$ in the sense of varifolds in $M$. Then for each point $x_0 \in $supp$\|V\| \cap \partial M$, there are $n \in \mathbb{N}$, $\rho>0$ (both depending on $x_0$) and an embedded minimal surface $\Gamma$ in $M$ meeting $\partial M$ orthogonally with 
\begin{equation*}
V \llcorner_{B_\rho(x_0)} = n v(\Gamma)
\end{equation*}
where $v(\Gamma)$ is the varifold represented by $\Gamma$ with multiplicity one.
\end{theorem}

\subsection{Convergence of the minimizing sequence}

In this section, we prove the main regularity result (Theorem 7.3). 

\begin{proof}[Proof of Theorem 7.3]
Let $\{\Sigma_k\}$ be a minimizing sequence for the minimization problem $(\Sigma, \mathfrak{Is}_{\text{out}}(U))$ with $\Sigma_k \cap M$ converging weakly to a varifold $V$ in $M$ and $\Sigma_k$ intersects $\partial M$ transversally for each $k$. Using the same argument as in \cite{Meeks-Simon-Yau82}, we can assume that $(\Sigma_k)_0 = \emptyset$ for all $k$ (see the paragraph above Theorem 7.10 for definition of $(\Sigma_k)_0$) and $\Sigma_k$ is strongly $(\gamma,U)$-irreducible for all sufficiently large k, for some fixed $0<\gamma<\delta^2/9$. Furthermore, we have 
\begin{equation}
\mathcal{H}^2(\Sigma_k \cap M) \leq \inf_{\Sigma \in J_U(\Sigma_k)} \mathcal{H}^2(\Sigma \cap M) + \epsilon_k
\end{equation}
where $\epsilon_k \to 0$ as $k \to \infty$.

As noted before, interior regularity and regularity at fixed boundary have been discussed in \cite{Meeks-Simon-Yau82} and \cite{DeLellis-Pellandini10}. So we only have to prove regularity at the free boundary.

Let $x_0 \in $supp$(\|V\|) \cap U \cap \partial M$, and $\nu_0$ be the outward unit normal at $x_0 \in \partial M$. Define $x_1=\text{exp}_{x_0}(\epsilon \nu_0)$ be a point outside $M$ which is very close to $x_0$ (by choosing $\epsilon$ very small).  Let $\rho_0>0$ be chosen small so that all the geodesic balls $B_\rho(x_1)$ in $\tilde{M}$ are admissible open sets in the sense of Definition 7.1 for all $0<\rho \leq \rho_0$. Note that we have to move the center of the balls from $x_0$ to $x_1$ in order for $(v)$ of Definition 7.1 to hold. 

First of all, we want to show that $V$ is freely stationary in $U$. Let $X \in \chi_{\text{tan}}$ be supported in $U$, and $\{\varphi_s\}_{s \in (-\epsilon,\epsilon)}$ be the isotopy generated by $X$. By (7.12),
\begin{equation}
\mathcal{H}^2(\Sigma_k \cap M) \leq \mathcal{H}^2(\varphi_s(\Sigma_k) \cap M) +\epsilon_k 
\end{equation}
for all $k$. Note that $\varphi_s(\Sigma_k) \cap M = \varphi_s(\Sigma_k \cap M)$ since $\{\varphi_s\} \in \mathfrak{Is}_{\text{tan}}(U)$, take $k \to \infty$ in (7.13), we get
\begin{equation}
\|V\|(M) \leq \| (\varphi_s)_\sharp V\|(M)
\end{equation}
for all $s \in (-\epsilon,\epsilon)$. This shows that $V$ is stable with respect to $\mathfrak{Is}_{\text{tan}}$, so $V$ is freely stationary. Therefore, the monotonicity formula in \cite{Gruter-Jost86} can be applied to $V$.

By the coarea formula, we have 
\begin{equation*}
\int_{\rho-\sigma}^\rho \mathcal{H}^1(\Sigma_k \cap \partial B_s(x_1) \cap M) \; ds \leq \mathcal{H}^2(\Sigma_k \cap (B_\rho(x_1) \setminus B_{\rho-\sigma}(x_1)) \cap M)
\end{equation*}
for almost every $\rho \in (0,\rho_0)$ and every $\sigma \in (0,\rho)$, where $B_s(x_1)$ is the closed geodesic ball in $\tilde{M}$ of radius $s$ centered at $x_1$. Taking $\sigma=\rho/2$, the monotonicity formula in \cite{Gruter-Jost86} gives
\begin{equation*}
\int_{\rho/2}^\rho \mathcal{H}^1(\Sigma_k \cap \partial B_s(x_1) \cap M) \; ds \leq c \rho^2
\end{equation*}
for all sufficiently large $k$, where $c$ depends only on $M$ and $\tilde{M}$ and any upper bound for $\rho_0^{-2} (\|V\|(B_{\rho_0}(x_1))+\|V\|(\tilde{B}_{\rho_0}(x_1)))$ (recall that $\tilde{B}$ is the reflection of $B$ across $\partial M$, see \cite{Gruter-Jost86} for definition). Hence, we can find a sequence $\{\rho_k\} \subset (3\rho/4,\rho)$ such that $\Sigma_k$ intersects $\partial B_{\rho_k}(x_1)$ transversally and such that 
\begin{equation}
\mathcal{H}^1(\Sigma_k \cap \partial B_{\rho_k}(x_1) \cap M) \leq c \rho \leq c \eta \rho_0
\end{equation}
for all sufficiently large $k$, provided $\rho \leq \eta \rho_0$, where for the moment $\eta \in (0,1)$ is arbitrary. If $\eta$ is sufficiently small, we see from (7.15) that Theorem 7.10 is applicable. Hence, there are connected genus zero surfaces $D_k^{(1)},\ldots,D_k^{(q_k)} \subset \Sigma_k \cap M$ and for any $\alpha>0$, isotopies $\{ \varphi_t^{(k)}\}_{t \in [0,1]} \in \mathfrak{Is}_{\text{out}}(B_{\rho_0}(x_1))$ such that 
\begin{equation}
\partial D_k^{(j)} \setminus \partial M \subset \partial B_{\rho_k}(x_1),
\end{equation}
\begin{equation*}
\Sigma_k \cap B_{\rho_k}(x_1)\cap M = \left( \cup_{j=1}^{q_k} D_k^{(j)} \right) \cap B_{\rho_k}(x_1),
\end{equation*}
\begin{equation}
\mathcal{H}^2(\cup_{j=1}^{q_k}(\varphi_1^{(k)}(D_k^{(j)}) \setminus \partial D_k^{(j)}) \cap M \setminus (B_{\rho_k}(x_1) \setminus \partial B_{\rho_k}(x_1))) < \alpha,
\end{equation}
and
\begin{equation*}
\sum_{j=1}^{q_k} \mathcal{H}^2(D_k^{(j)}) \leq c \rho^2 \leq c \eta^2 \rho_0^2,
\end{equation*}
where $c$ is independent of $k, \eta$ and $\rho$. Since $\mathcal{H}^2(D_k^{(j)}) \leq c\eta^2 \rho_0^2$, we know that for $\eta$ sufficiently small, by the modified replacement lemma with free boundary (see Lemma 4.4 in \cite{Jost86}), there are connected genus zero surfaces $\tilde{D}_k^{(j)}$ contained in $M$ with 
\begin{equation}
\partial \tilde{D}_k^{(j)} \setminus \partial M = \partial D_k^{(j)} \setminus \partial M, \; \tilde{D}_k^{(j)} \setminus \partial \tilde{D}_k^{(j)} \subset B_{\rho_k}(x_1),
\end{equation}
and 
\begin{equation}
\mathcal{H}^2(\tilde{D}_k^{(j)}) \leq \mathcal{H}^2(D_k^{(j)}).
\end{equation}
Combining (7.12) and (7.19), using Lemma 7.12, (7.16), (7.17) and (7.18), $D^{(j)}_k$ is a minimizing sequence among all genus zero surfaces with fixed boundary $\partial D^{(j)}_k$ with any number of free boundaries on $\partial M$. By Theorem 7.13, we know that for each $x_0 \in $supp$(V) \cap \partial M$, there exist $n \in \mathbb{N}$ and $\rho>0$ and an embedded minimal surface $\Sigma$ meeting $\partial M$ orthongonally with $V=n\Sigma$ on $B_\rho(x_0)$. This finishes the proof of Theorem 7.3.
\end{proof}


\section{Regularity of outer almost minimizing varifolds}

In this section, we define the notion of \emph{good replacement property} for freely stationary varifolds and prove that if there exists sufficiently many replacements, then the varifold must be a smooth minimal surface with free boundary. In the second half, we will describe how to construct these replacements for outer almost minimizing varifolds.

\begin{definition}
Let $V \in \mathcal{V}(M)$ be a freely stationary varifold and $U \subset \tilde{M}$ be an open subset. We say that $V' \in \mathcal{V}(M)$ is a \emph{replacement} for $V$ in $U$ if and only if 
\begin{enumerate}
	\item $V'$ is freely stationary;
	\item $V'=V$ on $G(M \setminus U)$ and $\|V'\|(M)=\|V\|(M)$; 
	\item $V' \llcorner_{(U\cap M)}$ is (an integer multiple of) a smooth stable properly embedded (not necessarily connected) minimal surface $\Sigma \subset M$ meeting $\partial M$ orthogonally. Here, $\Sigma$ is stable means that the second variation is nonnegative with respect to isotopies $\{\varphi_t\}_{t \in (-\epsilon, \epsilon)}$ supported in $U$ and $\varphi_t(M)=M$ for all $t$.
	\end{enumerate}
\end{definition}

\begin{definition}
Let $V \in \mathcal{V}(M)$ be freely stationary and $U \subset \tilde{M}$ be an open subset. We say that $V$ has \emph{good replacement property in $U$} if and only if all the following hold:
\begin{itemize}
	\item[(a)] There is a positive function $r_1:U \to \mathbb{R}$ such that for every annulus An$_1 \in \mathcal{A}_{r_1(x)}(x)$, there is a replacement $V'$ for $V$ in An$_1$ such that (b) holds;
	\item[(b)] There is another positive function $r_2:U \to \mathbb{R}$ such that 
		\begin{itemize}
			\item[(i)] $V'$ has a replacement $V''$ in any An$_2 \in \mathcal{A}_{r_1(x)}(x)$, for the same $x$ and $r_1$ as in (a), and $V''$ satisfies (c) below;
			\item[(ii)] $V'$ has a replacement in any An$\in \mathcal{A}_{r_2(y)}(y)$ for any $y \in U$.
		\end{itemize} 
	\item[(c)] There is yet another positive function $r_3:U \to \mathbb{R}$ such that $V''$ has a replacement in any An$\in \mathcal{A}_{r_3(z)}(z)$ for any $z \in U$. 
\end{itemize}
\end{definition}

The key result in this section is the following regularity theorem.

\begin{theorem}
If $V$ has good replacement property in an open set $U \subset \tilde{M}$, then $V$ is a smooth embedded minimal surface in $U \cap $int$(M)$ with smooth free boundary on $U \cap \partial M$. 
\end{theorem}

The proof of interior regularity can be found in \cite{Colding-DeLellis03}. Therefore, we will only prove regularity at the free boundary. To prove Theorem 8.3, we first state a generalization of two lemmas from \cite{Colding-DeLellis03} adapted to the free boundary setting.

\begin{lemma}
Let $U$ be an open set of $\tilde{M}$ and $x \in \partial M \cap U$. Then, there exists some $R>0$ such that $B_R(x) \cap \partial U =\emptyset$ and there does not exist any $W \in \mathcal{V}(M)$ with support contained in $B_R(x)$ and freely stationary in $U$.
\end{lemma}

\begin{proof}
We prove the lemma in the case $U=B_1(0) \subset \mathbb{R}^3$ and $M \cap U=B_1(0) \cap \{z \geq 0\}$. The proof for the general case is similar. Fix $x \in \partial M \cap U$, let $R>0$ be small enough (smaller than the convexity radius of $x$ in the general case) such that $B_R(x) \subset \subset B_1(0)$. If the lemma is false, then there exists a freely stationary varifold $W$ in $U \cap M$ which is supported in $B_R(x)$. Choose $r>0$ to be the smallest radius such that $B_r(x)$ contains the support of $W$, then the support of $W$ either touches $\partial B_r(x) \cap M$ at a point $y$ in the interior of $M$ or on the boundary $\partial M$. Either case cannot happen by the maximum principle (see Lemma B.1 of \cite{Colding-DeLellis03} or Theorem 1 of \cite{White10}). Note that in the general case, we have to perturb by a projection (see \cite{Gruter-Jost86}) the variation field to make it tangent to $\partial M$, but the perturbation would be small if we choose $R>0$ small enough, so we would still arrive at a contradiction.
\end{proof}

Given a varifold $V$ and a point $y \in \tilde{M}$, we let $T(y,V)$ be the set of varifold tangents of $V$ at $y$ (see section 42 of \cite{Simon83}).

\begin{lemma}
Let $x \in \tilde{M}$, and $V \in \mathcal{V}(M)$ be a freely stationary integer rectifiable varifold. Assume $T$ is the subset of supp$(\|V\|)$ defined by
\begin{equation*}
T=\{ y \in \text{supp}(\|V\|) : T(y,V) \text{ consists of a plane transversal to } \partial B_{d(x,y)}(x) \}.
\end{equation*}
If $\rho$ is less than the injectivity radius inj$(\tilde{M})$ of $\tilde{M}$, then $T$ is dense in supp$(\|V\|) \cap B_\rho(x) \setminus \partial M$.
\end{lemma}

\begin{proof}
The proof is similar to the proof of Lemma B.2 in \cite{Colding-DeLellis03}.
\end{proof}

The next proposition tells us what we can say about the freely stationary varifold if there exists one replacement.

\begin{proposition}
Let $U \subset \tilde{M}$ be open and $V \in \mathcal{V}(M)$ be a freely stationary varifold in $U$. If there exists a positive function $r:U \to \mathbb{R}$ such that $V$ has a replacement in any annulus An$\in \mathcal{A}_{r(x)}(x)$. 

Then, $V$ is integer rectifiable in $U \cap M$. Moreover, if $x \in $supp$(\|V\|) \cap (\text{int}(M)\cap U)$, then $\theta(x,V) \geq 1$ and any tangent cone to $V$ at $x$ is an integer multiple of a plane; if $x \in $supp$(V) \cap (\partial M \cap U)$, then $\theta(x,V) \geq \frac{1}{2}$ and any tangent cone to $V$ at $x$ is an integer multiple of a half-plane orthogonal to $T_x \partial M$.
\end{proposition}

\begin{proof}
Fix an $x \in $ supp$(\|V\|)\cap(\partial M \cap U)$. Since $V$ is freely stationary, the monotonicity formula (3.2) from \cite{Gruter-Jost86} gives $R>0$ and a constant $C>0$ (depending only on $M$ and $\tilde{M}$) such that for all $y$ in some neighborhood of $\partial M$ in $M$, and $0<\sigma<\rho<R$,
\begin{equation*} 
\frac{\|V\|(B_\sigma(y))}{\sigma^2} \leq C \frac{\|V\|(B_\rho(y))}{\rho^2} .
\end{equation*}
We can assume that $R$ is small enough so that Lemma 8.4 is satisfied. Choose $\rho<r(x)/2$ and so that $4\rho<R$ is smaller than the convexity radius of $\tilde{M}$. Since $2\rho<r(x)$, there is a replacement $V' \in \mathcal{V}(M)$ for $V$ in the annulus An$(x,\rho,2\rho)$. First of all, $V' \neq 0$ on An$(x,\rho,2\rho)$. Otherwise, since $V=V'$ in $B_\rho(x)$, we have $x \in$supp$(\|V'\|)$ and there would be a $\sigma \leq \rho$ such that $V'$ touches $\partial B_\sigma(x)$ from the interior, i.e. $\sigma=\max_{y \in \text{supp}(\|V'\|)} d(y,x)$. This would contradict Lemma 8.4. Therefore, $V'$ is a non-empty smooth surface in An$(x,\rho,2\rho)$ which meets $\partial M$ orthogonally, and so there is some $y \in $An$(x,\rho,2\rho) \setminus \partial M$ with $\theta(V',y) \geq 1$. By the monotonicity formula, and notice that $y \notin \partial M$, 
\begin{equation*} 
\frac{\|V\|(B_{4\rho}(x))}{16 \rho^2} =  \frac{\|V'\|(B_{4\rho}(x))}{16 \rho^2}  \geq  \frac{\|V'\|(B_{2\rho}(y))}{16 \rho^2} \geq \frac{\pi}{4C}. 
\end{equation*}
For $x \in $supp$(\|V\|) \cap $int$(M) \cap U$, the usual monotonicity formula for stationary varifold gives a similar lower bound. Hence, $\theta(x,V)$ is bounded uniformly from below on supp$(\|V\|)$, applying the rectifiability theorem, we conclude that $V$ is rectifiable.

The interior case was handled in \cite{Colding-DeLellis03}, and we know that $V$ is integer rectifiable. So it remains to prove the free boundary case in the proposition. Fix $x \in $supp$(\|V\|)\cap(\partial M \cap U)$, and a sequence $\rho_n \downarrow 0$ such that $V^x_{\rho_n}$ converges weakly to a tangent cone $C \in TV(x,V)$ which is stationary with respect to all variations tangential to $T_x \partial M$. By a change of coordinate, we can assume that $T_x \partial M$ has inward pointing normal $(0,0,1)$.  We will show that $C$ is an integer multiple of a half plane $H$. Since $C$ is freely stationary, $H$ must be orthogonal to $T_x\partial M$, hence contain $(0,0,1)$.

First, we place $V$ by $V'_n$ in the annulus An$(x,\rho_n/4, 3\rho_n/4)$ and set $W'_n=(T^x_{\rho_n})_\sharp V'_n$. After possibly passing to a subsequence, we can assume that $W'_n \to C'$ weakly, where $C'$ is a stationary varifold with respect to tangential variations. By the definition of replacements, we have 
\begin{equation*}
C'=C \hspace{1cm} \text{ in } \mathcal{B}_{\frac{1}{4}} \cup \text{An}(0,\frac{3}{4},1), 
\end{equation*}
and
\begin{equation} 
\|C'\|(\mathcal{B}_\rho)=\|C\|(\mathcal{B}_\rho) \hspace{1cm} \text{ for } \rho \in (0,\frac{1}{4}) \cup (\frac{3}{4},1),
\end{equation}
where $\mathcal{B}_s$ is the ball the radius $s$ in $\mathbb{R}^3$ centered at the origin. Since $C$ is a cone, using (8.1), we have 
\begin{equation*} 
\frac{ \|C'\|(\mathcal{B}_\sigma)}{\sigma^2}=\frac{\|C'\|(\mathcal{B}_\rho)}{\rho^2} \text{ for all } \sigma, \rho \in (0,\frac{1}{4}) \cup (\frac{3}{4},1). 
\end{equation*}
Hence, the stationarity of $C'$ and the monotonicity formula imply that $C'$ is also a cone. By Theorem 3.5, $W'_n$ converges to a stable properly embedded minimal surface in An$(x,1/4,3/4)$, with respect to variation fields in $\chi_{\text{tan}}$. This means that $C' \llcorner $An$(x,1/4,3/4)$ is an embedded minimal cone in the classical sense and hence is supported on a half disk containing the origin. The minimal cone is not the $x$-$y$ plane since each $W'_n$ meets $\partial M$ orthogonally. This forces $C'$ and $C$ to coincide and be an integer multiple of the same half plane perpendicular to $T_x \partial M$.
\end{proof}

We now give the proof of the main result (Theorem 8.3) in this section.

\begin{proof}[Proof of Theorem 8.3] 
Again, the interior regularity is covered in section 6 of \cite{Colding-DeLellis03}, so we only prove the free boundary regularity. 
Fix $x \in $supp$(\|V\|) \cap \partial M \cap U$. Choose $\rho$ small such that $\rho<r(x)/2$ and $2\rho$ is less than $R$ given in Lemma 8.4 and the convexity radius of $\tilde{M}$, by the good replacement property (a), we can find a replacement $V'$ for $V$ in the annulus An$(x,\rho,2\rho)$. Let $\Sigma'$ be the stable minimal surface given by $V'$ in An$(x,\rho,2\rho)$. For any $t \in (\rho, 2\rho)$ and $s \in (0,\rho)$, by the good replacement property (b)(i), we can find a replacement $V''$ of $V'$ in An$(x,s,t)$. Let $\Sigma''$ be the stable minimal surface given by $V''$ in An$(x,s,t)$. 

First, we choose some $t \in (\rho,2\rho)$ such that $\Sigma'$ intersects $\partial B_t(x)$ transversally. Such a $t$ exists because $\Sigma'$ is a smooth surface in the annulus An$(x,\rho,2\rho)$. Next, we show that $\Sigma' \cap $An$(x,t,2\rho)$ can be glued to $\Sigma'' \subset $An$(x,s,t)$ smoothly. It has already been shown that they glue together smooth in the interior in \cite{Colding-DeLellis03}, so it suffices to show that they also glue together smoothly along the free boundary.

Fix a point $y \in \Sigma' \cap \partial B_t(x) \cap \partial M$, and a sufficiently small radius $r$ so that $\Sigma' \cap B_r(y)$ is a half-disk orthogonal to $\partial M$ and $\gamma=\Sigma' \cap \partial B_t(x) \cap B_r(y)$ is a smooth arc perpendicular to $\partial M$. As in \cite{Colding-DeLellis03}, we can assume, without loss of generality, that $B_r(y)$ is the unit ball $\mathcal{B}$ in $\mathbb{R}^3$ centered at origin, $\partial M \cap B_r(y)=\{z_2=0\} \cap \mathcal{B}$, $M \cap B_r(y)=\{z_2 \geq 0\} \cap \mathcal{B}$. Moreover, $\partial B_t(x) \cap B_r(y)=\{z_1=0\} \cap \mathcal{B}$. Suppose $\Sigma' \cap B_r(y)$  is the graph of a smooth function $g(z_1,z_2)$ defined on $\{z_2 \geq 0, z_3=0\} \cap \mathcal{B}$. Hence, $\gamma=\{(0,z_2,g(0,z_2):z_2 \geq 0\}$.

The replacement $V''$ consists of $\Sigma'' \cup (\Sigma' \setminus B_t(x))$ in $B_r(y)$. By Proposition 8.6, using the fact that $V''$ satisfies (c) in Definition 8.2, $T(y,V'')$ consists of a family of (integer multiples) of half-planes orthogonal to $\{z_1=0\} \cap \mathcal{B}$ (in other words, they contain the vector $(0,1,0)$). Since $\Sigma'$ is regular and transversal to $\{z_1=0\}$, each half plane $P \in T(y,V'')$ coincides with the half plane $T_y\Sigma'$ in $\{z_1<0\}$. Therefore, $T(y,V'')=\{T_y\Sigma'\}$. Now,following the argument in \cite{Colding-DeLellis03}, we obtain a function $g''(z_1,z_2) \in C^1(\{z_1 \geq 0,z_2 \geq 0\})$ such that 
\begin{equation*}
\Sigma'' \cap B_r(y)=\{(z_1,z_2,g''(z_1,z_2)):z_1 >0, z_2 \geq 0\},
\end{equation*}
\begin{equation*}
g''(0,z_2)=g'(0,z_2) \text{ and } \nabla g''(0,z_2)=\nabla g'(0,z_2) \hspace{1cm} \text{for all }z_2 \geq 0.
\end{equation*}
Since $\Sigma'$ and $\Sigma''$ meets $\partial M$ orthogonally, we have the free boundary condition $\nabla g'(z_1,0)=(0,0)$ for all $z_1 \leq 0$ and $\nabla g''(z_1,0)=(0,0)$ for all $z_1>0$. By the reflection principle, one obtain a continuous function $G$ defined on the unit disk $\mathcal{D}=\{z_3=0\} \cap \mathcal{B}$ such that $G$ is smooth and satisfies the minimal surface equation on the punctured disk $\mathcal{D} \setminus 0$. Hence by standard interior regularity for second order uniformly elliptic PDE, $G$ is smooth across the origin.

By the maximum principle Lemma 8.4, we have shown that for any $s<\rho$, $\Sigma'$ can be extended to a surface $\Sigma_s$ in An$(x,s,2\rho)$ such that if $s_1<s_2<\rho$, then $\Sigma_{s_1}=\Sigma_{s_2}$ in An$(x,s_2,2\rho)$. Thus, $\Sigma=\cup_{0<s<\rho} \Sigma_s$ is a stable minimal surface with free boundary on $\partial M$ and $\overline{\Sigma} \setminus \Sigma \subset \partial B_{2\rho}(x) \cup \partial M \cup \{x\}$. Next, we show that $V$ coincides with $\Sigma$ in $B_\rho(x) \setminus \{x\}$. Recall that $V=V'$ in $B_\rho(x)$. Fix any $y \in $ supp$(\|V\|) \cap B_\rho(x) \setminus \{x\}$ and set $s=d(x,y)$. Since $\Sigma$ meets $\partial M$ orthogonally, $\mathcal{H}^2(\Sigma \cap \partial M)=0$, so we can assume $y \in $int$(M)$ and $T(y,V)$ consists of a multiple of a plane $\pi$ transversal to $\partial B_s(x)$ (by Lemma 8.5), then we know that $y \in \Sigma$ as in \cite{Colding-DeLellis03}. Therefore, (2) in the definition of replacement implies that $V=\Sigma$ on $B_\rho(x)$. 

It remains to show that $x$ is a removable singularity for $\Sigma$. By Proposition 8.6, every $C \in T(x,V)$ is a multiple of a half-plane orthogonal to $T_x \partial M$. Following \cite{Colding-DeLellis03}, for $\rho$ sufficiently small, there exists natural numbers $N(\rho)$ and $m_i(\rho)$ such that 
\begin{equation*}
\Sigma \cap \text{An}(x,\rho/2,\rho) = \cup_{i=1}^{N(\rho)} m_i(\rho) \Sigma^i_\rho
\end{equation*}
where each $\Sigma^i_\rho$ is a Lipschitz graph over a planar half-annulus, with the Lipschitz constants uniformly bounded independent of $\rho$. Hence, we get $N$ minimal punctured half-disks $\Sigma^i$ with 
\begin{equation*}
\Sigma \cap B_\rho(x) \setminus \{x\} = \cup_{i=1}^N m_i \Sigma^i.
\end{equation*}
By Allard regularity for stationary  varifolds with free bounday (\cite{Gruter-Jost86}), we see that $x$ is a removable singularity for each $\Sigma^i$. Finally, by the Hopf boundary lemma for uniformly elliptic second order PDE, $N$ must be one. This completes the proof of Theorem 8.3.
\end{proof}

To finish the proof of Proposition 4.11, it remains to construct replacements for limits of outer almost minimizing min-max sequences. Let $V$ be as in Proposition 4.10 and fix an annulus An$\in \mathcal{A}_{r(x)}(x)$. Set 
\begin{equation*}
\mathfrak{Is}_j(\textrm{An})=\{ \{\varphi_s\} \in \mathfrak{Is}_{\text{out}}(\textrm{An})  :\mathcal{H}^2(\varphi_s(\Sigma^j) \cap M) \leq \mathcal{H}^2(\Sigma_j \cap M)+\frac{1}{8j} \; \forall s \in [0,1]\}
\end{equation*} 

\begin{lemma}
For each $j$, suppose we have a minimizing sequence $\{\Sigma^{j,k}\}_{k \in \mathbb{N}}$ for the problem $(\Sigma^j,\mathfrak{Is}_j(\textrm{An}))$ that converges weakly to a varifold $V^j$.

Then, $V_j$ is a stable minimal surface in $\textrm{An}$ with free boundary on $\partial M$. Moreover, any $V^*$ which is the limit of a subsequence of $\{V^j\}$ is a replacement for $V$ (in the sense of Definition 8.1).
\end{lemma}

\begin{proof}
The proof of the second assertion is exactly as that in Proposition 7.5 in \cite{Colding-DeLellis03}. So we only prove the first assertion here. Without loss of generality, we assume that $j=1$, and we write $V'$, $\Sigma^k$ and $\Sigma$ in place of $V^j$, $\Sigma^{j,k}$ and $\Sigma^j$ respectively. Clearly $V'$ is stationary and stable in $\textrm{An}$, by its minimizing property. Thus, we only need to prove regularity. The proof follows exactly as in \cite{Colding-DeLellis03} except that we are using Theorem 3.5, Theorem 7.3 and the following result (see Lemma 7.6 in \cite{Colding-DeLellis03}), which can be proved by a rescaling argument as in \cite{Colding-DeLellis03}:

\textit{Fact:} Let $x \in \textrm{An}$, and assume that $\{\Sigma^k\}$ is minimizing for the problem $(\Sigma,\mathfrak{Is}_1(\textrm{An}))$. Then, there exists $\epsilon>0$ such that for $k$ sufficiently large, the following holds:
\begin{itemize}
	\item[(Cl)] For any $\{\varphi_s\} \in \mathfrak{Is}_{\text{out}}(B_\epsilon(x))$ with $\mathcal{H}^2(\varphi_1(\Sigma^k) \cap M) \leq \mathcal{H}^2(\Sigma^k \cap M)$, there exists another isotopy $\{\phi_s\} \in \mathfrak{Is}_{\text{out}}(B_\epsilon(x))$ such that $\varphi_1=\phi_1$ and 
	\begin{equation*}
		\mathcal{H}^2(\phi_s(\Sigma^k) \cap M) \leq \mathcal{H}^2(\Sigma^k \cap M) + \frac{1}{8} \text{ for all } s\in [0,1].
	\end{equation*}
\end{itemize}
Moreover, $\epsilon$ can be chosen so that (Cl) holds for any sequence $\{\tilde{\Sigma}^k\}$ which is minimizing for the problem $(\Sigma,\mathfrak{Is}_1(\textrm{An}))$ and with $\Sigma^j=\tilde{\Sigma}^j$ on $\tilde{M} \setminus \overline{B}_\epsilon(x)$.

\end{proof}

We end this section with a proof of Proposition 4.11.

\begin{proof}[Proof of Proposition 4.11]
We will apply Theorem 8.3. From Lemma 8.7 above, We know that  in every annulus An $\in \mathcal{A}_{r(x)}(x)$ there is a replacement $V^*$ for $V$. We need to show that $V$ satisfies (a), (b) and (c) in Definition 8.2, with $r=r_1$. 

\end{proof}


\section{Genus bound}

In this section, we observe that a result in De Lellis-Pellandini \cite{DeLellis-Pellandini10} which controls the topological type of the minimal surface constructed by min-max arguments continue to hold in the case of free boundary. (We will assume that all surfaces in a sweepout is orientable as in De Lellis-Pellandini \cite{DeLellis-Pellandini10}.) The proof is exactly the same as in De Lellis-Pellandini \cite{DeLellis-Pellandini10}. One only has to note that a compact smooth surface $\Gamma$ has genus $g$ if and only if the image of the map $r:H_1(\Gamma ; \mathbb{Z}) \to H_1(\Gamma,\partial \Gamma ; \mathbb{Z})$ is $\mathbb{Z}^{2g}$ when $\Gamma$ is orientable, or the image is $\mathbb{Z}^{g-1} \times \mathbb{Z}_2$ if $\Gamma$ is non-orientable. The lifting lemma (Proposition 2.1 in De Lellis-Pellandini \cite{DeLellis-Pellandini10}) is still valid and hence the proof goes through.

\begin{theorem}
Let $\Sigma^j=\Sigma^j_{t_j} \cap M$ be a sequence which is almost minimizing in sufficiently small annuli which intersects $\partial M$ transversally for all $j$ and $V$ be the varifold limit of   $\Sigma^j$ as $j \to \infty$. Write $V=\sum_{i=1}^N n_i \Gamma^i$ where $\Gamma^i$ are connected components of $\Sigma$, counted without multiplicity and $n_i$ are positive. Let $\mathcal{O}$ be the set of those $\Gamma^i$ which are orientable and $\mathcal{N}$ be those which are non-orientable. Then 
\begin{equation}
\sum_{\Gamma^i \in \mathcal{O}} \mathfrak{g}(\Gamma^i) + \frac{1}{2} \sum_{\Gamma^i \in \mathcal{N}} (\mathfrak{g}(\Gamma^i)-1) \leq \mathfrak{g}_0 = \liminf_{j \uparrow \infty} \liminf_{\tau \to t_j} \mathfrak{g}(\Sigma^j_\tau).
\end{equation}
where $\mathfrak{g}(\Gamma)$ denotes the genus of a smooth compact surface $\Gamma$ (possibly with boundary).
\end{theorem}

On the other hand, we note that it is impossible to get a similar bound on the connectivity  (i.e. number of free boundary components) of the minimal surface. Corollary 1.3 is then a direct consequence of Theorem 1.1 and Theorem 9.1 is the following. Noting that there is no closed minimal surface in $\mathbb{R}^3$.

\begin{corollary}
Any smooth compact domain in $\mathbb{R}^3$ contains a non-trivial properly embedded minimal surface $\Sigma$ with non-empty boundary which is a free boundary solution and such that 
\begin{itemize}
	\item[(i)] either $\Sigma$ is an orientable genus zero surface, i.e. a disk with holes;
	\item[(ii)] or $\Sigma$ is a non-orientable genus one surface, i.e. a M\"{o}bius band with holes.
\end{itemize}
\end{corollary}

This follows from the observation that any such domain can be swept out by surfaces with genus zero. In fact, we can generalize the result to arbitrary 3-manifold. Recall that for any orientable closed 3-manifold $M$, the \emph{Heegaard genus} of $M$ is the smallest integer $g$ such that $M=\Sigma_1 \cup \Sigma_2$, where $\Sigma_1 \cap \Sigma_2$ is an orientable surface of genus $g$ and each $\Sigma_i$, $i=1,2$, is a handlebody of genus $g$. For manifolds with boundary, we make the following definition.

\begin{definition}
Let $M$ be a compact 3-manifold with boundary. We define the \emph{filling genus} of $M$ to be the smallest integer $g$ such that there exists a smooth embedding of $M$ into a closed orientable 3-manifold $\tilde{M}$ with Heegaard genus $g$.
\end{definition}

Since any closed 3-manifold with Heegaard genus $g$ has a non-trivial sweepout by surfaces with genus as most $g$. The min-max construction on the saturation of such a sweepout together with the genus bound (Theorem 9.1) above give Theorem 1.2 as a corollary.

\begin{corollary}[Theorem 1.2]
Any smooth compact orientable 3-manifold $M$ with boundary $\partial M$ with filling genus $g$ contains a nonempty properly embedded smooth minimal surface $\Sigma$ with free boundary and the genus of $\Sigma$ is at most $g$ if it is orientable; and at most $2g+1$ if it is not orientable.
\end{corollary}

        

\appendix


\section{Proof of Lemma 3.1}

We now present a proof of Lemma 3.1.

\begin{proof}
Since $V_i$ is a weakly convergent sequence, $\|V_i\|(\tilde{M})$ is bounded. Hence $\|V_i \llcorner_M \|(\tilde{M})$ is also bounded. After passing to a subsequence (we use the same index $i$ for our convenience), $V_i \llcorner_M $ converges weakly to some $W \in \mathcal{V}(\tilde{M})$. Since $\mathcal{V}(M)$ is closed in $\mathcal{V}(\tilde{M})$, we have $W \in \mathcal{V}(M)$. We will show that $W=V \llcorner_M$. This clearly proves our lemma, since any subsequence of $V_i \llcorner_M $ has another subsequence converging weakly to $V \llcorner_M$.

First, we claim that $W \leq V \llcorner_M$, i.e. $W(f) \leq V \llcorner_M(f)$ for any nonnegative continuous function $f$ on $G(\tilde{M})$. Since $G(M)$ is a closed subset of $G(\tilde{M})$, there exists a decreasing sequence of continuous functions $\phi_k$ on $G(\tilde{M})$ with $0 \leq \phi_k \leq 1$, $\phi_k = 1$ on $G(M)$ and $\phi_k$ converges pointwise to the characteristic function $\chi_{G(M)}$ of $G(M)$. As $V_i$ converges weakly to $V$ in $\tilde{M}$, we have for each $k$, 
\begin{equation}
\lim_{i \to \infty} \int_{G(\tilde{M})} \phi_k f dV_i = \int_{G(\tilde{M})} \phi_k f dV. 
\end{equation}
Since $\phi_k$ and $f$ are nonnegative, for each $i$ and $k$,
\begin{equation} 
\int_{G(\tilde{M})} \phi_k f dV_i \llcorner_M \leq  \int_{G(\tilde{M})} \phi_k f dV_i 
\end{equation}
Holding $k$ fixed and taking $i \to \infty$ in (A.2), by (A.1), we have
\begin{equation*}
\int_{G(\tilde{M})} \phi_k f dW \leq \int_{G(\tilde{M})} \phi_k f dV. 
\end{equation*}
Since $W$ is supported in $M$ and $\phi_k=1$ on $M$ for each $k$, we get
\begin{equation*}
W(f)=\int_{G(M)} f dW=\int_{G(\tilde{M})} \phi_k f dW \leq \int_{G(\tilde{M})} \phi_k f dV. 
\end{equation*}
Now, since $\phi_k f$ converges pointwise to $\chi_{G(M)} f$ monotonically, by the monotone convergence theorem, 
\begin{equation*} 
\lim_{k \to \infty}  \int_{G(\tilde{M})} \phi_k f dV=\int_{G(\tilde{M})} \chi_M f dV=\int_{G(M)} f dV=V\llcorner_M(f). 
\end{equation*}
This proves our claim that $W \leq V \llcorner_M$.

Now, we want to show that $W=V \llcorner_M$. Since we already have the inequality $W \leq V \llcorner_M$. It suffices to show that $\|W\|(M)=\|V\llcorner_M\|(M)$. It follows from the assumption that $\|V_i\|(M)$ converges to $\|V\|(M)$ as $i \to \infty$,
\begin{equation*} 
\|W\|(M)=\lim_{i \to \infty} \|V_i \llcorner_M\|(M)=\lim_{i \to \infty} \|V_i\|(M) =\|V\|(M)= \|V\llcorner_M\|(M).
\end{equation*}
Therefore, the proof of Lemma 3.1 is completed.
\end{proof}


\section{The perturbation lemma}

We end this section by proving a technical perturbation lemma (Lemma 4.9). First, we prove a lemma which says that if we use a ``small'' isotopy to deform a surface, its area would not increase by too much.

\begin{lemma} Let $V \in \mathcal{V}(\tilde{M})$ be a varifold in $\tilde{M}$. Suppose we have a smooth vector field $X \in \chi_{\text{out}}$, and let $\{\varphi_s\}_{s \in [0,1]}$ be the outward isotopy in $\mathfrak{Is}_{\text{out}}$ generated by $X$. Then,
\begin{equation*}
\|(\varphi_1)_\sharp V \|(M) \leq \|V\|(M) \; e^{\|X\|_{C^1}}.
\end{equation*}
Here, $\|X\|_{C^1}$ denotes the $C^1$-norm of the vector field $X$ as a smooth map $X: \tilde{M} \to T\tilde{M}$.
\end{lemma}

\begin{proof}
By the first variation formula of area in $\tilde{M}$,
\begin{equation*}
|\delta V(X)| \leq \left| \int_{G(\tilde{M})} \text{div}_\pi X(x) \; dV(x,\pi) \right| \leq \|V\|(\tilde{M}) \|X\|_{C^1}.
\end{equation*}
Hence, if we let $f(s)=\|(\varphi_s)_\sharp V\|(\tilde{M})$, then the inequality above is equivalent to
\begin{equation*}
|f'(s)| \leq f(s) \|X\|_{C^1}.
\end{equation*}
Integrating $s$ from $0$ to $1$, we get $f(1) \leq f(0) \;  e^{\|X\|_{C^1}}$. In other words,
\begin{equation*}
\|(\varphi_1)_\sharp V\|(\tilde{M}) \leq \|V\|(\tilde{M}) \; e^{\|X\|_{C^1}}.
\end{equation*}
Now, using this and the assumption that $X$ is an outward vector field,  
\begin{equation*}
\|(\varphi_1)_\sharp V\|(M) \leq \|(\varphi_1)_\sharp(V \llcorner_M)\|(\tilde{M}) \leq \|V \llcorner_M\|(\tilde{M}) \; e^{\|X\|_{C^1}}=\|V\|(M) \; e^{\|X\|_{C^1}},
\end{equation*}
where the first inequality holds because $\|(\varphi_1)_\sharp V \|(M)=\| (\varphi_1)_\sharp (V \llcorner_M) \| (M)$. This proves Lemma B.1.
\end{proof}

We are now ready to prove the perturbation lemma.

\begin{lemma}
Given any sweepout $\{\Sigma_t\}_{t \in [0,1]} \in \Lambda$, and any $\epsilon >0$, there exists a continuous sweepout $\{\Sigma_t'\}_{t \in [0,1]} \in \Lambda$ such that 
\begin{equation*}
\mathcal{F}(\{\Sigma'_t\}) \leq \mathcal{F}(\{\Sigma_t\})+ \epsilon
\end{equation*} 
\end{lemma}

\begin{proof}
By 2.6.2(d) of Allard \cite{Allard72} and Lemma 3.1 in this paper, it suffices to construct $\{\Sigma'_t\}$ such that 
\begin{equation}
\mathcal{H}^2(\Sigma'_t \cap \partial M)=0
\end{equation}
for all $t \in [0,1]$. This would imply that $\{\Sigma'_t\}$ is a continuous sweepout. One might hope to perturb $\Sigma_t$ using an outward isotopy so that it is transversal to $\partial M$ for all $t$. However, it is impossible, in general, to find a \emph{smooth} family of outward isotopies such that all the perturbed surfaces are transversal to $\partial M$. On the other hand, we could find one so that all but finitely many $\Sigma_t$ is transversal to the boundary $\partial M$ after perturbation, and for those finitely many exceptions,  there are only finitely many points at which the perturbed $\Sigma_t$ meets the boundary $\partial M$ non-transversally. This would certainly imply (B.1).

For $\theta>0$ sufficiently small (to be chosen later), the signed distance function $d=d(\cdot,\partial M)$ is a smooth function on the open tubular neighborhood $U_\theta(\partial M)=\{x \in \tilde{M}:  |d(x,\partial M)| < \theta\}$ of $\partial M$. (We take $d$ to be nonnegative for points in $M$.) Let $\nu$ be the inward pointing unit normal to $\partial M$ with respect to $M$ (This is globally defined on $\partial M$ even when $M$ is not orientable). For $\theta>0$ small enough, we have a diffeomorphism
\begin{equation*}
g(x,s)=\text{exp}_x(s \nu(x)): \partial M \times (-\theta,\theta) \to U_\theta(\partial M).
\end{equation*}

We will need the following \emph{parametric Morse theorem}: If $f_t: N \to \mathbb{R}$ is a one-parameter family of smooth functions on a compact (possibly with boundary) manifold $N$ with $t \in [0,1]$, and $f_0$, $f_1$ are Morse functions, then there exists a smooth one-parameter family $F_t: N \to \mathbb{R}$ such that $F_0=f_0$, $F_1=f_1$, and $F$ is uniformly close to $f$ in the $C^k$-topology on functions $N \times [0,1] \to \mathbb{R}$. Furthermore, $F_t$ is Morse at all but finitely many $t$, at a non-Morse time, the function has only one degenerate critical point, corresponding to the birth/death transition.

The relationship between Morse functions and transversality can be seen as follows. Let $\Sigma$ be a closed surface in $\tilde{M}$. Consider the function $d$ restricted on $\Sigma \cap U_\theta(\partial M)$, if $d:\Sigma \cap U_\theta(\partial M) \to \mathbb{R}$ is a Morse function, then $\Sigma$ intersects the level sets $\{x \in M: d=c\}$ transversally except possibly at the critical points of $d$, which is only a finite set. In particular, we have $\mathcal{H}^2(\Sigma \cap \partial M)=0$. If $d$ is not a Morse function, we approximate it by a Morse function $d_\phi$ in $C^k$ norm in $\Sigma \cap U_\theta(\partial M)$. Without loss of generality, we can also assume that $d_\phi \leq d$ everywhere. Let $\phi \leq 0$ be a smooth extension of the function $d_\phi-d$ to $U_\theta(\partial M)$ (by Whitney's extension theorem), in such a way that $\|\phi\|_{C^k}$ is very small. Then, if we consider the outward vector field
\begin{equation*}
X(x)=\chi(x) \phi(x) \nabla d(x)
\end{equation*}
where $0 \leq \chi \leq 1$ is a smooth cutoff function on $\tilde{M}$ such that $\chi=1$ on $U_{\theta/2}(\partial M)$, $\chi=0$ outside $U_\theta(\partial M)$ and $|\nabla \chi| \leq 4/\theta$. Let $\{\varphi_s\}_{s \in [0,1]}$ be the outward isotopy generated by $X$. Then it is clear that $\mathcal{H}^2(\varphi_1(\Sigma) \cap \partial M)=0$ according to the discussion above. 

The perturbation can be carried out for each $\Sigma_t$ in the sweepout. Hence, we want to choose $\phi_t \leq 0$ on $U_\theta(\partial M)$ which are small in $C^k$ norm such that, after perturbation, $\Sigma_t$ intersects the boundary $\partial M$ at a set of $\mathcal{H}^2$- measure zero. The only complication is that we have to choose $\phi_t$ which depends smoothly on $t$. This is where we need the parametric Morse theorem.

By Lemma B.1, for every $C>0$ and $\epsilon>0$, there exists $\delta=\delta(C,\epsilon)>0$ sufficiently small (for example, take $\delta < \ln (1+\epsilon/2C)$) such that whenever $\|X\|_{C^1} < \delta$, 
\begin{equation}
\|(\varphi_1)_\sharp V\|(M) \leq \|V\|(M)+\frac{\epsilon}{2}
\end{equation}
for all $V \in \mathcal{V}(\tilde{M})$ with $\|V\|(M) \leq C$.

Fix a sweepout $\{\Sigma_t\}_{t \in [0,1]} \in \Lambda$ and $\epsilon>0$ as in the hypothesis, since $\{\Sigma_t\}$ is a continuous family of varifolds in $\tilde{M}$, there exists a constant $C>0$ such that $\mathcal{H}^2(\Sigma_t \cap M) \leq \mathcal{H}^2(\Sigma_t) \leq C$ for all $t \in [0,1]$. For this $\epsilon$ and $C$, choose $\delta>0$ so that (B.2) holds. Moreover, assume $\theta>0$ is always sufficiently small so that $d$ is a smooth function on $U_\theta(\partial M)$ and $g$ is a diffeomorphism.

Now, we would like to apply the parametric Morse theorem to the family of smooth functions $d_t:\Sigma_t \cap \overline{U}_\theta(\partial M) \to \mathbb{R}$. However, there is a little technical difficulty because $\Sigma_t \cap \overline{U}_\theta(\partial M)$ are not all diffeomorphic to each other. Recall that in the definition of a sweepout, we have two finite sets, $T \subset [0,1]$ and $P \subset \tilde{M}$, at which singularities occur. First of all, we argue that we can assume $P \cap \overline{U}_\theta(\partial M) = \emptyset$. 

Suppose $P \cap \partial M \neq \emptyset$. Since $P$ is just a finite set, there exists $0<\rho < \theta/2$ such that $\partial M_\rho \cap P = \emptyset$. Define an outward vector field $X \in C^\infty_{\text{out}}(\tilde{M},T\tilde{M})$ by 
\begin{equation*}
X(x)=-\rho \chi(x) \nabla d (x),
\end{equation*}
where $0\leq \chi \leq 1$ is a smooth cutoff function on $\tilde{M}$ such that $\chi=1$ on $U_\rho(\partial M)$, $\chi=0$ outside $U_\theta(\partial M)$ and $|\nabla \chi| \leq 4/\theta$. Let $K>0$ be a constant (independent of $\theta$) so that $|\nabla^2 d| \leq K$ on $\overline{U}_\theta(\partial M)$. Hence, for $\rho>0$ sufficiently small (depending on $\theta$ and $K$), we can make $\|X\|_{C^1} < \delta$. Let $\{\varphi_s\}_{s \in [0,1]}$ be the isotopy generated by $X$, then by (B.2), 
\begin{equation*}
\mathcal{H}^2(\varphi_1(\Sigma_t) \cap M) \leq \mathcal{H}^2(\Sigma_t \cap M)+\frac{\epsilon}{2}.
\end{equation*}
Moreover, $\varphi_1(\partial M_\rho)=\partial M$. Therefore, replacing $\{\Sigma_t\}$ by $\{\varphi_1(\Sigma_t)\}$ if necessary, we can assume that $P \cap \partial M=\emptyset$.

As $P$ is a finite set, we can further assume that $\theta$ is small enough so that $P \cap \overline{U}_\theta (\partial M)=\emptyset$. By the definition of a sweepout, there exists a partition $0=t_0<t_1<\cdots<t_k=1$ of the interval $[0,1]$ such that on each subinterval $[t_{i-1},t_i]$, $i=1,\ldots,k$, there exists an open neighborhood $U_i$ of $\partial M$ contained in $U_\theta(\partial M)$ such that $\Sigma_t \cap U_i$ are all diffeomorphic for $t \in [t_{i-1},t_i]$. 

Now, for any $\delta'>0$ (to be specified later), since Morse functions are dense in $C^k$-topology, for each $i=0,1,\ldots,k$, we can approximate the smooth function $d_{t_i}:\Sigma_{t_i} \cap U_i \to \mathbb{R}$ by a Morse function $\hat{d}_{t_i}:\Sigma_{t_i} \cap U_i \to \mathbb{R}$, with $\hat{d}_{t_i} \leq d_{t_i}$ and $\|d_{t_i}-\hat{d}_{t_i}\|_{C^1}< \delta'$. Let $\phi_{t_i} \leq 0$ be a smooth extension of $\hat{d}_{t_i} - d_{t_i}$ to $U_i$ such that $\|\phi_{t_i}\|_{C^1} \leq \delta'$ for all $i$. Using the parametric Morse Theorem, we can construct a smooth family of smooth functions $\phi_t \leq 0$ on $U_i$, $t \in [t_{i-1},t_i]$ such that $d_t+\phi_t $ is a Morse function on $\Sigma_t \cap U_i$ except for finitely many $t$'s there is only one degenerate critical point. We can also assume that $\phi_t$ is uniformly small in $C^1$-norm on $U_\theta(\partial M)$. Putting these intervals together, we have a piecewise smooth one-parameter family of smooth functions $\phi_t$, defined on $U$ for some neighborhood $U$ of $\partial M$, such that $d_t+\phi_t$ are Morse except at finitely many times. Since Morse functions form an open set in the space of all smooth functions in the $C^\infty$-topology, and the family is Morse at each $t_i$, we can smooth out the family, keeping it Morse except at finitely many times away from $t_i$'s. Assume $\theta$ is chosen small enough such that $U_\theta(\partial M) \subset U$. In summary, we have a smooth one-parameter family of smooth non-positive functions $\{\phi_t\}$ on $U$ such that $\|\phi_t\|_{C^1} < \delta'$. For each $t \in [0,1]$, let $\{\varphi_t(s)\}_{s \in [0,1]}$ be the outward isotopy generated by the outward vector field $X_t \in C^\infty_{\text{out}}(\tilde{M},T\tilde{M})$ defined as 
\begin{equation*}
X_t(x)=\phi_t(x) \chi(x) \nabla d(x), 
\end{equation*} 
where $0 \leq \chi \leq 1$ is a smooth cutoff function on $\tilde{M}$ so that $\chi=1$ on $U_{\theta/2}(\partial M)$, $\chi=0$ outside $U_\theta(\partial M)$ and $|\nabla \chi| \leq 2/\theta$. Let $\{\varphi_t(s)\}_{s \in [0,1]}$ be the isotopy in $\mathfrak{Is}^{\text{out}}$ generated by $X_t$. Take $\Sigma'_t=\varphi_t(1)(\Sigma_t)$, then $\{\Sigma'_t\} \in \Lambda$. We claim that $\{\Sigma'_t\}$ is the competitor we want.

First of all, by choosing $\delta'>0$ sufficiently small, we can make $\|X_t\|_{C^1} < \delta$ for all $t \in [0,1]$. Hence, we have from (B.2) that
\begin{equation*}
\mathcal{H}^2(\Sigma'_t \cap M) \leq \mathcal{H}^2(\Sigma_t \cap M) + \frac{\epsilon}{2}.
\end{equation*}
Moreover, since $d:\Sigma_t' \cap U_\theta(\partial M) \to \mathbb{R}$ agrees with $d_t+\phi_t$ for all $t \in [0,1]$, by our construction, we have $\Sigma'_t \cap \partial M$ consists of at most finitely many points for all $t \in [0,1]$. Therefore, we have 
\begin{equation*}
\mathcal{H}^2(\Sigma'_t \cap \partial M)=0. 
\end{equation*}
This completes the proof of Lemma B.2.
\end{proof}

\bibliographystyle{amsplain}
\bibliography{references}

\providecommand{\bysame}{\leavevmode\hbox to3em{\hrulefill}\thinspace}
\providecommand{\MR}{\relax\ifhmode\unskip\space\fi MR }
\providecommand{\MRhref}[2]{%
  \href{http://www.ams.org/mathscinet-getitem?mr=#1}{#2}
}
\providecommand{\href}[2]{#2}
\begin{thebibliography}{10}

\bibitem{Allard72}
William~K. Allard, \emph{On the first variation of a varifold}, Ann. of Math.
  (2) \textbf{95} (1972), 417--491. \MR{0307015 (46 \#6136)}

\bibitem{Almgren-Simon79}
Frederick~J. Almgren, Jr. and Leon Simon, \emph{Existence of embedded solutions
  of {P}lateau's problem}, Ann. Scuola Norm. Sup. Pisa Cl. Sci. (4) \textbf{6}
  (1979), no.~3, 447--495. \MR{553794 (81d:49025)}

\bibitem{Birkhoff17}
George~D. Birkhoff, \emph{Dynamical systems with two degrees of freedom},
  Trans. Amer. Math. Soc. \textbf{18} (1917), no.~2, 199--300. \MR{1501070}

\bibitem{Colding-DeLellis03}
Tobias~H. Colding and Camillo De~Lellis, \emph{The min-max construction of
  minimal surfaces}, Surveys in differential geometry, {V}ol.\ {VIII}
  ({B}oston, {MA}, 2002), Surv. Differ. Geom., VIII, Int. Press, Somerville,
  MA, 2003, pp.~75--107. \MR{2039986 (2005a:53008)}

\bibitem{Colding-Minicozzi11}
Tobias~Holck Colding and William~P. Minicozzi, II, \emph{A course in minimal
  surfaces}, Graduate Studies in Mathematics, vol. 121, American Mathematical
  Society, Providence, RI, 2011. \MR{2780140}

\bibitem{Courant40}
R.~Courant, \emph{The existence of minimal surfaces of given topological
  structure under prescribed boundary conditions}, Acta Math. \textbf{72}
  (1940), 51--98. \MR{0002478 (2,61a)}

\bibitem{DeLellis-Pellandini10}
Camillo De~Lellis and Filippo Pellandini, \emph{Genus bounds for minimal
  surfaces arising from min-max constructions}, J. Reine Angew. Math.
  \textbf{644} (2010), 47--99. \MR{2671775 (2011g:53012)}

\bibitem{Ulrich-Hildebrandt-Sauvigny10}
Ulrich Dierkes, Stefan Hildebrandt, and Friedrich Sauvigny, \emph{Minimal
  surfaces}, second ed., Grundlehren der Mathematischen Wissenschaften
  [Fundamental Principles of Mathematical Sciences], vol. 339, Springer,
  Heidelberg, 2010, With assistance and contributions by A. K{{\"u}}ster and R.
  Jakob. \MR{2566897}

\bibitem{Ulrich-Hildebrandt-Tromba10}
Ulrich Dierkes, Stefan Hildebrandt, and Anthony~J. Tromba, \emph{Regularity of
  minimal surfaces}, second ed., Grundlehren der Mathematischen Wissenschaften
  [Fundamental Principles of Mathematical Sciences], vol. 340, Springer,
  Heidelberg, 2010, With assistance and contributions by A. K{{\"u}}ster.
  \MR{2760441}

\bibitem{Fraser00}
Ailana~M. Fraser, \emph{On the free boundary variational problem for minimal
  disks}, Comm. Pure Appl. Math. \textbf{53} (2000), no.~8, 931--971.
  \MR{1755947 (2001g:58026)}

\bibitem{Fraser02}
\bysame, \emph{Minimal disks and two-convex hypersurfaces}, Amer. J. Math.
  \textbf{124} (2002), no.~3, 483--493. \MR{1902886 (2003d:53104)}

\bibitem{Goldhorn-Hildebrandt70}
Karlheinz Goldhorn and Stefan Hildebrandt, \emph{Zum {R}andverhalten der
  {L}{\"o}sungen gewisser zweidimensionaler {V}ariationsprobleme mit freien
  {R}andbedingungen}, Math. Z. \textbf{118} (1970), 241--253. \MR{0279699 (43
  \#5420)}

\bibitem{Gruter-Jost86a}
M.~Gr{{\"u}}ter and J.~Jost, \emph{On embedded minimal disks in convex bodies},
  Ann. Inst. H. Poincar{\'e} Anal. Non Lin{\'e}aire \textbf{3} (1986), no.~5,
  345--390. \MR{868522 (88f:49029)}

\bibitem{Gruter-Jost86}
Michael Gr{{\"u}}ter and J{{\"u}}rgen Jost, \emph{Allard type regularity
  results for varifolds with free boundaries}, Ann. Scuola Norm. Sup. Pisa Cl.
  Sci. (4) \textbf{13} (1986), no.~1, 129--169. \MR{863638 (89d:49048)}

\bibitem{Hildebrandt-Nitsche79}
S.~Hildebrandt and J.~C.~C. Nitsche, \emph{Minimal surfaces with free
  boundaries}, Acta Math. \textbf{143} (1979), no.~3-4, 251--272. \MR{549778
  (82e:49059)}

\bibitem{Hildebrandt69}
Stefan Hildebrandt, \emph{Boundary behavior of minimal surfaces}, Arch.
  Rational Mech. Anal. \textbf{35} (1969), 47--82. \MR{0248650 (40 \#1901)}

\bibitem{Jager70}
Willi J{{\"a}}ger, \emph{Behavior of minimal surfaces with free boundaries},
  Comm. Pure Appl. Math. \textbf{23} (1970), 803--818. \MR{0266067 (42 \#976)}

\bibitem{Jost86}
J{{\"u}}rgen Jost, \emph{Existence results for embedded minimal surfaces of
  controlled topological type. {I}}, Ann. Scuola Norm. Sup. Pisa Cl. Sci. (4)
  \textbf{13} (1986), no.~1, 15--50. \MR{863634 (89m:58040)}

\bibitem{Jost86a}
\bysame, \emph{Existence results for embedded minimal surfaces of controlled
  topological type. {II}}, Ann. Scuola Norm. Sup. Pisa Cl. Sci. (4) \textbf{13}
  (1986), no.~3, 401--426. \MR{881099 (89m:58041)}

\bibitem{Almgren65}
Frederick J.~Almgren Jr., \emph{The theory of varifolds}, Mimeographed notes,
  Princeton University, 1965.

\bibitem{Lewy51}
Hans Lewy, \emph{On mimimal surfaces with partially free boundary}, Comm. Pure
  Appl. Math. \textbf{4} (1951), 1--13. \MR{0052711 (14,662e)}

\bibitem{Li11}
Man Chun~(Martin) Li, \emph{On a free boundary problem for embedded minimal
  surfaces and instability theorems in manifolds with positive isotropic
  curvarure}, Ph.D. thesis, Stanford University, 2011.

\bibitem{Lin-Yang02}
Fanghua Lin and Xiaoping Yang, \emph{Geometric measure theory---an
  introduction}, Advanced Mathematics (Beijing/Boston), vol.~1, Science Press,
  Beijing, 2002. \MR{2030862 (2005a:28001)}

\bibitem{Meeks-Simon-Yau82}
William Meeks, III, Leon Simon, and Shing~Tung Yau, \emph{Embedded minimal
  surfaces, exotic spheres, and manifolds with positive {R}icci curvature},
  Ann. of Math. (2) \textbf{116} (1982), no.~3, 621--659. \MR{678484
  (84f:53053)}

\bibitem{Meeks-Yau80}
William~H. Meeks, III and Shing~Tung Yau, \emph{Topology of three-dimensional
  manifolds and the embedding problems in minimal surface theory}, Ann. of
  Math. (2) \textbf{112} (1980), no.~3, 441--484. \MR{595203 (83d:53045)}

\bibitem{Meeks-Yau82a}
William~W. Meeks, III and Shing~Tung Yau, \emph{The existence of embedded
  minimal surfaces and the problem of uniqueness}, Math. Z. \textbf{179}
  (1982), no.~2, 151--168. \MR{645492 (83j:53060)}

\bibitem{Nitsche69}
Johannes C.~C. Nitsche, \emph{The boundary behavior of minimal surfaces.
  {K}ellogg's theorem and {B}ranch points on the boundary}, Invent. Math.
  \textbf{8} (1969), 313--333. \MR{0259766 (41 \#4399a)}

\bibitem{Nitsche70}
\bysame, \emph{Minimal surfaces with partially free boundary. {L}east area
  property and {H}{\"o}lder continuity for boundaries satisfying a chord-arc
  condition}, Arch. Rational Mech. Anal. \textbf{39} (1970), 131--145.
  \MR{0266068 (42 \#977)}

\bibitem{Nitsche85}
\bysame, \emph{Stationary partitioning of convex bodies}, Arch. Rational Mech.
  Anal. \textbf{89} (1985), no.~1, 1--19. \MR{784101 (86j:53013)}

\bibitem{Pitts81}
Jon~T. Pitts, \emph{Existence and regularity of minimal surfaces on
  {R}iemannian manifolds}, Mathematical Notes, vol.~27, Princeton University
  Press, Princeton, N.J., 1981. \MR{626027 (83e:49079)}

\bibitem{Pitts-Rubinstein86}
Jon~T. Pitts and J.~H. Rubinstein, \emph{Existence of minimal surfaces of
  bounded topological type in three-manifolds}, Miniconference on geometry and
  partial differential equations ({C}anberra, 1985), Proc. Centre Math. Anal.
  Austral. Nat. Univ., vol.~10, Austral. Nat. Univ., Canberra, 1986,
  pp.~163--176. \MR{857665 (87j:49074)}

\bibitem{Sacks-Uhlenbeck81}
J.~Sacks and K.~Uhlenbeck, \emph{The existence of minimal immersions of
  {$2$}-spheres}, Ann. of Math. (2) \textbf{113} (1981), no.~1, 1--24.
  \MR{604040 (82f:58035)}

\bibitem{Schoen-Simon-Yau75}
R.~Schoen, L.~Simon, and S.~T. Yau, \emph{Curvature estimates for minimal
  hypersurfaces}, Acta Math. \textbf{134} (1975), no.~3-4, 275--288.
  \MR{0423263 (54 \#11243)}

\bibitem{Schoen83}
Richard Schoen, \emph{Estimates for stable minimal surfaces in
  three-dimensional manifolds}, Seminar on minimal submanifolds, Ann. of Math.
  Stud., vol. 103, Princeton Univ. Press, Princeton, NJ, 1983, pp.~111--126.
  \MR{795231 (86j:53094)}

\bibitem{Schoen06}
\bysame, \emph{Minimal submanifolds in higher codimension}, Mat. Contemp.
  \textbf{30} (2006), 169--199, XIV School on Differential Geometry
  (Portuguese). \MR{2373510 (2009h:53137)}

\bibitem{Simon83}
Leon Simon, \emph{Lectures on geometric measure theory}, Proceedings of the
  Centre for Mathematical Analysis, Australian National University, vol.~3,
  Australian National University Centre for Mathematical Analysis, Canberra,
  1983. \MR{756417 (87a:49001)}

\bibitem{Smith82}
F.~Smith, \emph{On the existence of embedded minimal 2-spheres in the 3-sphere,
  endowed with an arbitrary {R}iemannian metric}, Ph.D. thesis, University of
  Melbourne, 1982, PhD thesis, Supervisor: Leon Simon.

\bibitem{Struwe84}
M.~Struwe, \emph{On a free boundary problem for minimal surfaces}, Invent.
  Math. \textbf{75} (1984), no.~3, 547--560. \MR{735340 (85a:58019)}

\bibitem{Taylor77}
Jean~E. Taylor, \emph{Boundary regularity for solutions to various capillarity
  and free boundary problems}, Comm. Partial Differential Equations \textbf{2}
  (1977), no.~4, 323--357. \MR{0487721 (58 \#7336)}

\bibitem{White10}
Brian White, \emph{The maximum principle for minimal varieties of arbitrary
  codimension}, Comm. Anal. Geom. \textbf{18} (2010), no.~3, 421--432.
  \MR{2747434}

\bibitem{Ye91}
Rugang Ye, \emph{On the existence of area-minimizing surfaces with free
  boundary}, Math. Z. \textbf{206} (1991), no.~3, 321--331. \MR{1095757
  (92g:58023)}

\end{thebibliography}

\end{document}